\documentclass[noinfoline]{imsart}

\RequirePackage[OT1]{fontenc}
\RequirePackage{amsmath,amssymb,amsfonts,amsthm,bbm}
\RequirePackage{natbib}

\usepackage{enumerate}
\usepackage{amssymb}
\usepackage{mathrsfs}
\usepackage{amsthm}
\usepackage[utf8]{inputenc}
\RequirePackage[OT1]{fontenc}
\usepackage{enumerate}
\usepackage{amssymb}
\usepackage{amsfonts} 
\usepackage{amsmath}
\usepackage{accents}
\usepackage{color}
\usepackage{graphicx}
\usepackage{caption}
\usepackage{subcaption}
\usepackage{dsfont}
\usepackage{multirow}
\usepackage{mathtools}
\usepackage[usenames,dvipsnames,svgnames,table]{xcolor}
\usepackage{color,graphicx}
\usepackage{algorithmicx,algpseudocode, algorithm}

\definecolor{darkgreen}{rgb}{0,0.6,0.3}
\definecolor{red}{rgb}{0.7,0.15,0.15}
\definecolor{darkblue}{rgb}{0,0,0.6}
\usepackage{hyperref}
\usepackage{float}
\hypersetup{
	colorlinks=true,
	linkcolor=red,
	citecolor= darkgreen,
	urlcolor = black}
\usepackage[capitalise]{cleveref}

\DeclareMathAlphabet{\mathpzc}{OT1}{pzc}{m}{it}

\startlocaldefs

\newtheorem{theorem}{Theorem}[section]
\newtheorem{lemma}[theorem]{Lemma}

\newtheorem{prop}[theorem]{Proposition}
\newtheorem{corollary}[theorem]{Corollary}

\newtheorem{definition}[theorem]{Definition}
\newtheorem{assumption}{Assumption}			
\theoremstyle{definition}

\newtheorem{defn}[theorem]{Definition}
\theoremstyle{remark}
\newtheorem{remark}[theorem]{Remark}

\makeatother
\newcommand{\Pm}{\mathbb{P}}
\newcommand{\PS}{(\Omega, \mathcal{F},\left(\mathcal{F}_{t}\right)_{t\geq 0}, \mathbb{P})}

\newcommand{\pa}{U_{n}}

\newcommand{\cf}{\widehat{\phi}_{n}}
\newcommand{\ut}{\textbf{u}_{n}}

\newcommand{\uti}{\tilde{\textbf{u}}_{n}}
\newcommand{\f}{\widehat{\phi}_{n}(\uti)}
\newcommand{\fit}{\widehat{\phi}_{n}(\ut)}

\newcommand{\R}{\mathbb{R}}

\newcommand{\E}{\mathbb{E}}
\newcommand{\ch}{\mathds{1}}
\newcommand{\var}{\textsf{Var}}
\newcommand{\re}{\textsf{Re}}
\newcommand{\im}{\textsf{Im}}

\newcommand{\dotprod}{\langle \textbf{u}, X_{t}\rangle}
\newcommand{\dotprodu}{\langle \textbf{u}, X_{1}\rangle }

\newcommand{\ent}{N_{[ \cdot ]}(\epsilon, \mathbb{G})}
\newcommand{\ei}{J_{[ \cdot ]}(\delta, \mathbb{G})}
\newcommand{\g}{\mathbb{G}}

\newcommand{\class}{\mathcal{L}^{r}_{M}}

\numberwithin{equation}{section}

\begin{document}

	\numberwithin{equation}{section}

		\begin{frontmatter}
			\title{A Lepski\u\i\/-type stopping rule for the covariance estimation of multi-dimensional L\'evy processes}
			\runtitle{{A Lepski\u\i\/-type stopping rule for the covariance estimation}}
			
			\begin{aug}
				\author{\fnms{Katerina}  \snm{Papagiannouli}
					\ead[label=e1]{papagiai@hu-berlin.de}}
				
				
				\runauthor{K. Papagiannouli}
				
				\affiliation{Humboldt-Universit\"at zu Berlin}

				\address{ Institut f\"ur Mathematik\\Humboldt-Universit\"at zu Berlin\\ Unter den Linden 6\\10099 Berlin\\Germany\\ \printead{e1}}
			\end{aug}

		\begin{abstract}
We suppose that a L\'evy process is observed at discrete time points. Starting from an asymptotically minimax family of estimators for the continuous part of the L\'evy Khinchine characteristics, i.e., the covariance, we derive a data-driven parameter choice for the frequency of estimating the covariance. We investigate a Lepski\u\i\/-type stopping rule for the adaptive procedure. Consequently, we use a balancing principle for the best possible data-driven parameter. The adaptive estimator achieves almost the optimal rate. Numerical experiments with the proposed selection rule are also presented.
	\end{abstract}
\end{frontmatter}
	\section{Introduction}
In recent years, the use of multi-dimensional L\'evy processes for modeling purposes has become very popular in many areas, especially in the field of finance (e.g. \cite{tankov2003financial}; see also \cite{sa} for a comprehensive study). The distribution of a L\'evy process is usually specified by its characteristic triplet (drift, Gaussian component, and L\'evy measure) rather than by the distribution of its independent increments. Indeed, the exact distribution of these increments is most often intractable or without closed formula. For this reason, an important task is to provide estimation methods for the characteristic triplet.

Such estimation methods depend on the way observations are performed. In our model, two-dimensional L\'evy process $ \textbf{X}_{t} $ is observed at high frequency, i.e., the time between two consecutive observations is $\frac{1}{n}$. The characteristic function of such a two-dimensional L\'evy process is given by
\begin{equation}\label{cftwo}
\phi_{n}(\textbf{u}_{n}) := \E [\exp(i\langle \textbf{u}_{n}, X_{t}\rangle)] = \exp\bigg\{\frac{1}{n}\Psi(\textbf{u}_{n};\textbf{b}, C, F)\bigg\}, \quad  \mbox{$ \textbf{u}_{n}  \in \R^{2} $},
\end{equation}
where
\begin{equation}\label{exponentcf}
\begin{aligned}
\Psi(\textbf{u}_{n}) = \Psi(\textbf{u}_{n};\textbf{b}, C, F) =& i\left\langle \textbf{u}_{n},\textbf{b}\right\rangle-\frac{\left\langle C \textbf{u}_{n},\textbf{u}_{n}\right\rangle}{2}+\int_{\R^{2}}\big(\exp(i\left\langle \textbf{u}_{n}, \textbf{x}\right\rangle)\\
&-1-i\left\langle \textbf{u}_{n}, \textbf{x}\right\rangle \ch_{\left\{||\textbf{x}||_{\R^{2}}\leq 1\right\}}\big)F(d\textbf{x}),
\end{aligned}
\end{equation} 
$\textbf{b} \in \R^{2} $ is the drift, $ C = \begin{psmallmatrix}
C^{11}&C^{12}\\
C^{21}&C^{22}
\end{psmallmatrix} $ is the covariance matrix, and $F \in \mathcal{P}(\R^{2}) $ is the jump measure. The triplet $ (\textbf{b}, C, F) $ is called L\'evy Khinchine characteristic. By virtue of simplicity, we consider the characteristic function on the diagonal and concentrate primarily on a two-dimensional regime, but extensions to the general multi-dimensional setting are straightforward to obtain as well.\par 

Nonparametric inference from high-frequency data on the triplet of a L\'evy process has been considered by \cite{barndorff2002}, \cite{ait}, \cite{jacod2014remark}, \cite{bibinger2014estimating}, \cite{mancini2017truncated}, \cite{MR3825892}, and the references therein. In addition, minimax estimation of the covariance has been the subject of \cite{papagiannouli2020}. In this work, the author develops a family of covariance estimators $ \widehat{C}_{n}^{12}(U_{n}) $ to infer $ C^{12} $. Although this contribution proves that $ \widehat{C}^{12}_{n}(U_{n}) $ achieves minimax rates for the estimation of $ C^{12} $, this approach nevertheless presents a drawback insofar as $ U_{n} $ depends on a number of unknown parameters, such as the co-jump activity index $ r \in(0, 2]$. Co-jumps refer to the case when the underlying processes jump at the same time with the same direction. $ r $ refers to the Blumenthal-Getoor index for co-jumps. To overcome this shortcoming, a data-driven choice $ \widehat{U} $ is needed which ensures near-minimax rates for the estimation error.\par

A natural way to extend minimax theory to an adaptation theory is to construct estimators which simultaneously achieve near-minimax rates over many subsets of parameter space. Starting with the work of \cite{Lep90}, the design of minimax-adaptive estimators for linear functionals has been widely covered in the literature, e.g. \cite{efromovich1994adaptive} and \cite{birge}. Lepski\u\i\/ designed a strategy for choosing a data-dependent parameter which uses only differences between estimators. His stopping rule considered only the monotonicity of the deterministic and stochastic errors. This method is widely applied in learning theory, where supervised learning algorithms depend on some tuning parameter, correct choice of which is crucial to ensure optimal performance.\par

Although it is no easy task, the implementation of Lepski\u\i\/-type stopping rule has been used in the literature as a recipe for adaptive procedures, e.g. \cite{de2010adaptive} and the references therein. What interests us particularly in the present context is the fact that we have to deal with the problem of adaptation to the unknown characteristic function appearing in the denominator of the stochastic error. A behavior which also occurs in the deconvolution problem, e.g. \cite{neumann1997effect}, \cite{comte2011data}, \cite{dattner2016adaptive} In our case, the unknown characteristic function in the denominator leads to the stochastic error behaving irregularly. In order to apply Lepski\u\i\/'s rule, it is crucial to overcome this irregular behavior. \par

The main contribution of the present work is to construct adaptive estimators and extend the minimax result obtained in \cite{papagiannouli2020}. We provide a remedy for the irregular behavior of the stochastic error. The unknown characteristic function in the denominator leads to a U-shaped stochastic error. This behavior prevents us from applying Lepski\u\i\/'s rule. So it is crucial to find an index for the oracle start of our parameter. As a result, a monotonically increasing bound for the stochastic error is constructed. Finally, the convergence rate of the adaptive estimator is proven to be near-minimax.\par

The remainder of the paper is organized as follows. \Cref{sec:characteristic} provides general results for the uniform control of the deviation of the empirical characteristic function on $ \R^{2} $, so that it also can be read as an independent contribution. \Cref{sec:adaptive} introduces Lepski\u\i's\/ strategy for devising a stopping rule algorithm for the parameter $U$. In \Cref{sec:stochasticerror}, we present theoretical guarantees for the adaptive estimation. Hence, we are able to construct a monotonically increasing upper bound for the stochastic error. In \Cref{sec:bp}, we devise a balancing principle for the optimal choice of $ U $ and present the convergence rates of the adaptive estimator. \Cref{sec:discussion} summarizes the results. A short illustration of the behavior of the estimator and stopping rules is then provided in \Cref{sec:experiments} by means of empirical simulations from synthetic data. Finally, proofs for \Cref{sec:characteristic} are given in \Cref{sec:proofs3}.

\section{Estimating the characteristic function}\label{sec:characteristic}
Here, we discuss technical tools which provide a uniform control of the deviations of the empirical characteristic function on $ \R^{2} $. The interesting point here is that the decay of the characteristic function is not assumed to be explicitly known but comes in by implication. To keep the exposition intuitive and free from technicalities, the proofs of lemmas have been postponed to \Cref{sec:proofs3}. Throughout this section, we use the letter $ C $ to denote a constant that may change from line to line.\par

For the sake of keeping the calculations simple, we will restrict ourselves to estimating the characteristic function on the diagonal. For this purpose, let us introduce the following definition.

\begin{definition}\label{diagonalset}
	We define the subsets of the diagonal as
	\begin{equation*}
	\begin{aligned}
	&\mathcal{A} := \{\textbf{u} \in \R^{2}: \textbf{u} = (U, U),  U \in \R\}\\
	&\tilde{\mathcal{A}} := \{\tilde{\textbf{u}} \in \R^{2}: \tilde{\textbf{u}} = (U, -U), U \in \R \}.
	\end{aligned}
	\end{equation*}
\end{definition}

Let a probability space $ \PS $ be given. We assume that $\textbf{X}_{t} = (X^{(1)}, X^{(2)}) $ is a bivariate L\'evy process observed at $ n $ equidistant time points $ \Delta, \ldots ,n \Delta = T $, where $ \Delta = \frac{i}{n} $ for $ i = 1, \ldots, n $ and $ T = 1 $. We denote by 
\begin{equation}\label{normcf}
C_{n}(\textbf{u}) := \frac{1}{\sqrt{n}}\bigg(\sum_{j= 1}^{n} e^{i\langle \textbf{u}, \Delta_{j}^{n}\textbf{X}\rangle} - \E[e^{i\langle \textbf{u}, \textbf{X}_{1/n}\rangle}]\bigg)
\end{equation}
the normalized empirical characteristic function process, where $ \textbf{u} \in \mathcal{A}$. For an appropriate weight function $ w:\R \to (0, 1]$, we consider 
\begin{equation}
\E\|C_{n}\|_{L_{\infty}(w)}:=\E \sup_{\textbf{u}\in \mathcal{A}}\big\{|C_{n}(\textbf{u})|w(U)\big\}.
\end{equation}
Recall that $ C_{n}(\textbf{u}) $ converges weakly to a Gaussian process if and only if $ \big\{\textbf{x }\to e^{i\langle \textbf{u}, \textbf{x}\rangle}, \textbf{u} \in \mathcal{A} \big\} $ is a functional Donsker class for $ \Pm$.\par

We start by defining a weight function that was introduced in \cite{neumann2009nonparametric} and is the key for the uniform convergence of the empirical characteristic function.
\begin{definition}\label{weight}
	For some $ \delta>0 $, let the weight function $ w $ be defined as
	\begin{equation*}
	w(U) := \big(\log(e + |U|)\big)^{-\frac{1}{2}-\delta}.
	\end{equation*}
\end{definition}
The above definition is meaningful under the following, rather general assumption concerning the characteristic function.

\begin{assumption}\label{quasimonoton}
	There is a function $ g  $ which is non-decreasing on $ \R^{-} $ and non-increasing on $ \R^{+} $. There exist positive constants $ C$  and $ C'$, such that
	\begin{equation*}
	\begin{aligned}
	&\forall \textbf{u} \in \mathcal{A}: C g (U)\leq|\phi_{n}(\textbf{u})|\leq C'g(U)\\
	&\forall \tilde{\textbf{u}} \in \mathcal{\tilde{A}}: C g (U)\leq|\phi_{n}(\tilde{\textbf{u}})|\leq C'g(U).
	\end{aligned}
	\end{equation*}
\end{assumption}

Some remarks are in order here: The following cases may be considered for the characteristic function.
\begin{enumerate}[(\textbf{a})]
	\item \textbf{Gaussian decay.} Under some boundedness condition for the covariance matrix and the activity of jumps, we can prove that
	\begin{equation*}
	|\phi_{n}(\textbf{u})|\geq e^{-\frac{CU^{2}}{2n}}, \qquad \mbox{\textbf{$ \forall \textbf{u}\in\mathcal{A} $}}.
	\end{equation*}
\end{enumerate}

\begin{enumerate}[(\textbf{b})] 
	\item \textbf{Exponential decay.} Here,the characteristic function $ \phi_{n} $ decays at most exponentially, that is, for some $ a>0 $, $ C>0 $,
	\begin{equation*}
	|\phi_{n}(\textbf{u})|\geq C e^{-a|U|/n}, \qquad \mbox{\textbf{$ \forall \textbf{u}\in\mathcal{A} $}}.
	\end{equation*}
\end{enumerate}
Examples of distributions with this property include normal inverse Gaussian and generalized tempered stable distributions. 
\begin{enumerate}[(\textbf{c})]
	\item \textbf{Polynomial decay.} In this case the characteristic function satisfies for some $ \beta\geq 0 $, $ C> 0$, 
	\begin{equation*}
	|\phi_{n}(\textbf{u})|\geq C(1+|U|)^{-\beta/n}, \qquad \mbox{$\forall \textbf{u}\in \mathcal{A}$}.
	\end{equation*}
\end{enumerate}
Typical examples for this property are the compound Poisson distribution, gamma distribution, and variance gamma distribution. Contrary to the properties formulated above, our reasoning does not rely on any semiparametric assumption about the shape of the characteristic function. The only thing needed is the quasi-monotonicity of Assumption \ref{quasimonoton} which is fairly general.

We receive the following result, extending Theorem 4.1 of \cite{neumann2009nonparametric} in two dimensions.

\begin{theorem}\label{ecf}
	Suppose that $ (X_{t})_{t\in \mathbb{N}} $ are i.i.d. random vectors in $ \R^{2} $ with $ \E|X_{1}|^{2+\gamma}<\infty$ for some $ \gamma >0 $, and let the weight function $ w$ be defined as in Definition \ref{weight}. Then
	\begin{equation*}
	\sup_{n\geq 1}\E \| C_{n}\|_{L_{\infty}(w)} <\infty.
	\end{equation*}
\end{theorem}

Let us mention that the logarithmic decay of the weight function $ w $ is in accordance with the well-known results of \cite{csorgHo1983long}, where 
\[
\lim\limits_{n\to \infty}C_{n}\left((T_{n}, T_{n})\right) = 0
\]
almost surely on intervals $ [-T_{n}, T_{n}] $ whenever $ \log T_{n}/n \to \infty$. We are now ready to prove a uniform bound for the deviation of the empirical characteristic function from the true one. First, we establish a Talagrand inequality using Lemma \ref{talagrand} from Appendix \ref{app}.
\begin{lemma}\label{countable}
	Let $ \mathcal{I}$ be some countable index set. Then for arbitrary $ \epsilon >0 $, there are positive constants $ c_{1}, c_{2} = c_{2}(\epsilon) $, such that for every $ \kappa>0 $ we obtain
	\begin{equation*}
	\begin{aligned}
	\Pm\bigg[\sup_{j\in \mathcal{I}}|\cf(\textbf{u}_{j})  - \phi_{n}(\textbf{u}_{j})|&\geq (1+\epsilon)\E\big[\sup_{j\in \mathcal{I}}|\cf(\textbf{u}_{j})- \phi_{n}(\textbf{u}_{j})|\big] +\kappa\bigg]\\
	&\leq 2 \exp\bigg(-n\bigg(\frac{\kappa^{2}}{c_{1}}\wedge \frac{\kappa}{c_{2}}\bigg)\bigg).
	\end{aligned}
	\end{equation*}
\end{lemma}

Now we introduce a logarithmic factor which is essential to proving uniformness on the diagonal. This comes at the cost of losing a logarithmic factor. 

\begin{lemma} \label{uniform}
	Let $ t>0 $ be given, and $ \mathcal{A} $ defined as in Definition \ref{diagonalset}. Then, for arbitrary $ \beta >0 $, there exists a constant $ C $, such that we have
	\begin{equation*}
	\Pm\bigg[\exists \textbf{u}\in \mathcal{A}: |\cf(\textbf{u}) - \phi_{n}(\textbf{u})|\geq t \bigg(\frac{\log{n}}{n}\bigg)^{1/2}(
	w(U))^{-1}\bigg]\leq Cn^{-\frac{(t-\beta)^{2}}{c_{1}}},
	\end{equation*}
	where the constant $  C$ depends on $ \delta  $ appearing in Definition \ref{weight} and $ c_{1} $ is the constant in Talagrand's inequality from Lemma \ref{countable}.
\end{lemma}
The statement of \Cref{uniform} holds for $\tilde{\textbf{u}} \in \tilde{A}$. A direct consequence of Lemma \ref{uniform} is that we can consider a favorable set for the deviation of the empirical characteristic function from the true one.
\begin{lemma}\label{set}
	For some $ p\geq 1/2 $ and $ \kappa \geq 4(\sqrt{pc_{1}} +\beta) $, let us consider the event
	\begin{equation*}
	\mathcal{E} := \bigg\{\forall \textbf{u} \in \mathcal{A}: |\cf(\textbf{u}) -\phi(\textbf{u})|\leq \frac{\kappa}{4} \bigg(\frac{\log 
		n}{n}\bigg)^{1/2}(w(U))^{-1}\bigg\}.
	\end{equation*}
	\begin{equation*}
	\tilde{\mathcal{E}} : = \bigg\{\forall \tilde{\textbf{u}} \in \tilde{\mathcal{A}}: |\cf(\tilde{\textbf{u}}) -\phi(\tilde{\textbf{u}})|\leq \frac{\kappa}{4} \bigg(\frac{\log 
		n}{n}\bigg)^{1/2}(w(U))^{-1}\bigg\}.
	\end{equation*}
	Thus, we have
	\[
	\Pm\bigg[\mathcal{E}^{\complement}\bigg]\leq Cn^{-p} \quad \mbox{and} \quad  \Pm\bigg[\tilde{\mathcal{E}}^{\complement}\bigg]\leq Cn^{-p}. 
	\]
\end{lemma}
Lemma \ref{uniform} and Lemma \ref{set} hold for $ \tilde{\textbf{u}} \in \mathcal{\tilde{A}} $ as well.


\subsection{Truncated characteristic function}\label{sec:tcf}
Here we present an extension of Lemma 2.1 in \cite{neumann1997effect}, which renders the point-wise control of the characteristic function in the denominator uniform on sets $ \mathcal{A} $. Now, we briefly discuss the idea of a truncated characteristic function presented in detail in \cite{neumann1997effect}. It is clear that the characteristic function $ \phi_{n}(\textbf{u}) $ can be estimated at each point $ \textbf{u}= (U, U) $ with the rate $ n^{-1/2} $. Hence, $ \cf(\textbf{u}) $ is a reasonable estimator of $ \phi_{n}(\textbf{u}) $, if $ |\phi_{n}(\textbf{u})|\gg n^{-1/2}$. The idea is to cut off the frequencies $ \textbf{u}$, for which $ |\phi_{n}(\textbf{u}) |\leq n^{-1/2}$.\par
First, we recall the key Lemma 2.1 from \cite{neumann1997effect}:
\begin{lemma}
	It holds that, for any $ p\geq 1$,
	\begin{equation*}
	\E \bigg( \bigg| \frac{1}{\tilde{\phi}_{n}(u)}-\frac{1}{\phi_{n}(u)}\bigg|^{2p}\bigg)\leq C\bigg(\frac{1}{|\phi_{n}(u)|^{2p}}\wedge \frac{n^{-p}}{|\phi_{n}(u)|^{4p}}\bigg),
	\end{equation*}
	where $ \frac{1}{\tilde{\phi}_{n}(u)}:= \frac{\mathds{1}\big(|\cf(u)|\geq n^{-1/2}\big)}{\widehat{\phi}_{n}(u)}$.
\end{lemma}
Neumann's result is for $ p=1 $, but the extension to any $ p $ is straightforward. See also \cite{neumann2009nonparametric}. The global threshold must be formulated in terms of $ \widehat{\phi}_{n}(\textbf{u}) $, so that the compact set is in fact random. The main difference with Neumann's truncated estimator lies in the fact that we introduce an additional logarithmic factor in the thresholding scheme. This logarithmic factor allows us to derive exponential inequalities, as we saw in Lemma \ref{uniform}.
\begin{definition}\label{truncatedest}
	Let the weight function $ w $ be given like in Definition \ref{weight}. For some positive constant $ \kappa $, set
	\begin{equation}
	\frac{1}{\tilde{\phi}_{n}(\textbf{u})}:= \begin{cases}
	\frac{1}{\widehat{\phi}_{n}(\textbf{u})}, & \text{if $ |\cf(\textbf{u})|\geq \kappa_{n}n^{-1/2}, $}\\
	\frac{1}{\kappa_{n}n^{-1/2}}, &  \textit{otherwise}
	\end{cases}
	\end{equation}
	where $ \kappa_{n} := \frac{\kappa}{2} (\log n)^{1/2}(w(U))^{-1} $.
\end{definition}

We can now use Lemma \ref{uniform} to assess the deviation of $ \frac{1}{\tilde{\phi}_{n}(\textbf{u})} $ from $ \frac{1}{\phi_{n}(\textbf{u})} $.

\begin{lemma}\label{nk}
	Suppose that for some $ p\geq 1/2 $ and $ \beta>0 $, we have $ \kappa\geq 2(\sqrt{pc_{1}} + \beta) $, where $ c_{1} $ is the constant in Talagrand's inequality. Then, for $ n>0 $ and a positive constant $ C $, we have
	\begin{equation*}
	\Pm\bigg[\exists \textbf{u} \in \mathcal{A}: \bigg|\frac{1}{\tilde{\phi}_{n}(\textbf{u})} - \frac{1}{\phi_{n}(\textbf{u})} \bigg|^{2}  > \bigg(\frac{9\kappa^{2}}{16}\frac{\log n(w(U))^{-2}n^{-1}}{|\phi_{n}(\textbf{u})|^{4}}\wedge \frac{1}{4}\frac{1}{|\phi_{n}(\textbf{u})|^{2}}\bigg)\bigg]\leq C n^{-p}.
	\end{equation*}
\end{lemma}

We are now in position to formulate a uniform bound on the diagonal, which is an immediate result of Lemma \ref{nk}. 

\begin{lemma}\label{unk}
	If the assumptions of Lemma \ref{nk} hold, then there is a constant $ C>0 $ depending on $ \kappa$, such that for $ n\geq 1 $
	\begin{equation*}
	\E \bigg[\sup_{\textbf{u}\in \mathcal{A}}\bigg|\frac{1}{\tilde{\phi}_{n}(\textbf{u})}-\frac{1}{\phi_{n}(\textbf{u})}\bigg|^{2} \bigg(\frac{\log n (w(U))^{-2}n^{-1}}{|\phi_{n}(\textbf{u})|^{4}}\wedge \frac{1}{|\phi_{n}(\textbf{u})|^{2}}\bigg)^{-1}\bigg]\leq C.
	\end{equation*}
\end{lemma}

Also, Lemma \ref{unk} can be extended to powers different from $ 2 $. We just need to substitute $ 2 $ with $ 2q $.\par

Note that an intermediate consequence of the preceding Lemma \ref{nk} is the following important corollary, which allows us to interchange between the empirical characteristic function and the true one with high probability.
\begin{corollary}\label{emptotrue}
	In the situation of the preceding statement, we have
	\begin{equation*}
	\Pm \bigg[\exists \textbf{u} \in \mathcal{A}: \bigg|\frac{1}{\tilde{\phi}_{n}(\textbf{u})}-\frac{1}{\phi_{n}(\textbf{u})} \bigg|>\frac{1}{2}\bigg|\frac{1}{\phi_{n}(\textbf{u})}\bigg|\bigg]\leq C n^{-p}.	
	\end{equation*}
\end{corollary}
It is in fact this version of the statement which will play an important role below. On the complement of the preceding event, we have with high probability
\begin{equation}\label{inequality}
\begin{aligned}
-\frac{1}{2} \bigg|\frac{1}{\phi_{n}(\textbf{u})}\bigg| &\leq -\bigg|\frac{1}{\phi_{n}(\textbf{u})}\bigg|+\bigg|\frac{1}{\tilde{\phi}_{n}(\textbf{u})}\bigg|\leq \frac{1}{2}\bigg|\frac{1}{\phi_{n}(\textbf{u})}\bigg|\\
\frac{1}{2} \bigg|\frac{1}{\phi_{n}(\textbf{u})}\bigg|& \leq \bigg|\frac{1}{\tilde{\phi}_{n}(\textbf{u})}\bigg|\leq \frac{3}{2}\bigg|\frac{1}{\phi_{n}(\textbf{u})}\bigg|.
\end{aligned}
\end{equation}
The statement of \Cref{emptotrue} and the above inequality hold for $ \tilde{\textbf{u}} \in \tilde{A}$.
\section{Adaptive parameter estimation}\label{sec:adaptive}

After recalling the statistical model, in this section we discuss the goal of this study. We aim to extend the minimax theory, from \cite{papagiannouli2020}, to an adaptation theory for the covariance estimator.

\subsection{Statistical model}
We observe a two-dimensional L\'evy process $(\textbf{X}_{t_{i}})_{t_{i}\geq 0}$ for $ i = 0,1,\ldots,n $ at equidistant time points $0 = t_{0}<t_{1}<\ldots<t_{n},  $ where $ t_{i} = \frac{i}{n} $. We consider the characteristic function (\ref{cftwo}) on the diagonal, i.e., $ \textbf{u}_{n}= (U_{n}, U_{n}) $, with characteristic triplet $ (\textbf{b}, C, F) $ with drift part $\textbf{b} \in \R^{2} $, covariance matrix $ C = \begin{psmallmatrix}
C^{11}&C^{12}\\
C^{21}&C^{22}
\end{psmallmatrix} $, and jump measure $F \in \mathcal{P}(\R^{2}) $.\par 
In what follows, we are in a nonparametric setting in which the process $ \textbf{X}_{t_{i}} $ belongs to the class $ \class $. Let us now recall this class.
\begin{defn}\label{class}
	For $M>0$ and $ r\in [0,2) $,  we define the class $\class$, the set of all L\'evy processes, satisfying
	\begin{equation}\label{assu}
	\|C\|_{\infty}+\int_{\R^{2}}\left(1\wedge |x_{1}x_{2}|^{r/2}\right) F(dx_{1},dx_{2}) < M,
	\end{equation}
	where $ \|C\|_{\infty} = \max(|C^{11}+C^{12}|, |C^{21}+C^{22}|) $ is the maximum of the row sums. In the second term $ r $ refers to the co-jump activity index of the jump components.
\end{defn}
For details and examples concerning this class we refer to Section 3 in \cite{papagiannouli2020}, where a minimax estimator for the covariance $ C^{12} $ is available. In addition, \cite{jacod2014remark} provide a minimax estimator for the marginals, i.e. $ C^{11}, C^{22} $. 
Given the empirical characteristic function of the increments $ \Delta \textbf{X}_{j} =\textbf{X}_{j/n}- \textbf{X}_{(j-1)/n} $
\begin{equation*}
\widehat{\phi}_{n}(\textbf{u}_{n}) := \frac{1}{n}\sum_{j = 1}^{n} e ^{i\langle \textbf{u}_{n}, \Delta_{j}^{n}\textbf{X}\rangle}, \quad \mbox{$ \textbf{u}_{n} \in \R^{2} $}
\end{equation*}
a spectral estimator is used:
\begin{equation*}
\widehat{C}^{12}_{n}(U_{n})=\frac{n}{2U_{n}^{2}}\left(\log \vert\hat {\phi}_{n}(\uti)\vert\mathds{1}(\hat{\phi}_{n}(\uti)\neq 0)- \log \vert\hat {\phi}_{n}(\ut)\vert\mathds{1}(\hat{\phi}_{n}(\ut)\neq 0)\right),
\end{equation*}
where $\textbf{u}_{n}= (U_{n}, U_{n}) $, $ \tilde{\textbf{u}}_{n} = (U_{n}, -U_{n})$. \par
A bias-variance type decomposition for the estimation error is available by Lemma 6.1 in \cite{papagiannouli2020}. We recall the Lemma without the proof.
\begin{lemma}\label{errordecomposition}
	The error bound for the estimation satisfies 
	\begin{equation}\label{biasvar}
	|\widehat{C}^{12}_{n}(U_{n}) - C^{12}|\leq |H_{n}(U_{n})| + |D(U_{n})|,
	\end{equation}
	where
	\begin{equation}
	D(U_{n})=\frac{n}{2U^{2}_{n}}\bigg(\log \vert\phi_{n}(\uti)\vert-\log \vert\phi_{n}(\ut)\vert\bigg)-C_{12},
	\end{equation}
	and
	\begin{equation}\label{st}
	H_{n}(U_{n})=-\frac{n}{2U^{2}_{n}}\bigg(\log \Bigl|\frac{\phi_{n}(\uti)}{\phi_{n}(\ut)}\Bigr|-\bigg(\log \Bigl|\frac{\f}{\fit}\Bigr|\bigg)\mathds{1}\left(\f \neq 0, \fit \neq 0\right)\bigg).
	\end{equation} 
	$ H_{n}(\cdot), D(\cdot) $ are the corresponding stochastic and deterministic errors.
\end{lemma}
The spectral estimator $ \widehat{C}^{12}_{n}(U_{n}) $ achieves minimax rates for the optimal parameter $ U_{n} $. For $ r\in[0,2) $, $ M$ defined as in Definition \ref{class} and for every $ 0 < \eta \leq 1 $, there is a constant $ A_{\eta}>0 $, and $ N_{\eta} $ such that for every $ n\geq N_{\eta} $
\begin{equation}\label{tight}
\Pm\Big[|\widehat{C}^{12}_{n}(U_{n}) - C^{12}|\leq w_{n}A_{\eta}\Big]\geq 1-\eta,
\end{equation}
where 
\begin{equation}\label{rates}
w_{n} = \begin{cases}
n^{-1/2} & \text{if $ r\leq 1 $}\\
(n\log n)^{\frac{r-2}{2}} & \text{if $ r>1 $}
\end{cases}
\end{equation}
are the minimax rates for the optimal parameter
\begin{equation}\label{optimal}
U_{n} = \begin{cases} \sqrt{n} &\text{$ r\leq 1 $}\\ 
\frac{\sqrt{(r-1)n\log n}}{\sqrt{M}} & \text{$r>1.$}\end{cases}
\end{equation} 
The error bound incurred by the spectral estimator in \Cref{errordecomposition} is the sum of two terms, i.e., the deterministic and stochastic error, with respect to the tuning parameter $ \pa $. The stochastic error displays behavior opposite to the deterministic error. The stochastic error tends to explode, however the deterministic error tends to zero as $ U_{n} $ grows. This observation and the fact that $ U_{n} $ depends on unknown parameters ($ r, M $) impose the need for a-posteriori choices of the parameter $U_{n}  $, which ideally are optimal in a well-defined sense. The goal is to derive a theoretical error bound for the adaptive estimator achieving almost the optimal rates.


\subsection{Lepski\u\i\/'s stopping rule}\label{sec:lepski}

In this section, we establish an adaptive choice for the parameter $ \pa $, as this is achieved by Lepski\u\i\/'s principle. Following Lepski\u\i\/'s principle, a ``stopping'' rule is designed to achieve adaptation for a class of minimax estimators. We use the following conventions for the notations. We denote by  $ \mathcal{U}$ the parameter space. 
We consider a suitable finite discretization $U_{0}<\ldots < U_{K} $ for our parameter. We set $ \widehat{C}^{12}_{n, j}:= \widehat{C}^{12}(U_{j})  $, i.e., we assign an estimator $ \widehat{C}^{12}_{n,j} $ for each $ U_{j} $. For each estimator $ \widehat{C}^{12}_{n,j} $ we set $ s_{n}(U_{j})$ to be the upper bound of the stochastic error $\E |H_{n}(U_{j})|$ for $ j= 0,1,\ldots, K $ .\par

Starting from a family of rate asymptotically minimax estimators $ \big\{\widehat{C}^{12}_{n}(U_{n})\big\} $, how can one get adaptation over the parameter space $ \mathcal{U}$, to find an optimal tuning parameter $ U_{j} $, which provides simultaneously minimax rates for the covariance over the sets $ [U_{0}, U_{K}]\subset\mathcal{U} $? \par

\begin{remark}\label{remark}
	In this paper we refer to the value $ U_{n} $ as the best choice and to the corresponding rate as the best possible rate. The rate will be optimal in a minimax sense since the bound we started from is tight (\ref{tight}).
\end{remark}

Let us first give a brief and simplified account of the classical Lepski\u\i\/ method adjusted to our problem. We use the results in Section 5.4 of \cite{markus}. The key idea is to test real-valued estimators $\widehat{C}^{12}_{n, 1}, \widehat{C}^{12}_{n, 2}, \dots, \widehat{C}^{12}_{n, j}  $, whose stochastic errors are increasing as the index is increasing and the bias is decreasing, for the hypotheses $ H_{j}: \widehat{C}^{12}_{n, 1} = \widehat{C}^{12}_{n, 2}= \cdots = \widehat{C}^{12}_{n, j}$. If we accept $ H_{1}, H_{2}, \dots ,H_{j} $, this means that $ \widehat{C}^{12}_{n, j+1}$ differs significantly from $ \widehat{C}^{12}_{n, 1} , \widehat{C}^{12}_{n, 2}, \dots,\widehat{C}^{12}_{n, j}$ so we reject $ H_{j+1} $. Further, we set $ \widehat{j} = j$. We summarize the above discussion in the following definition.

\begin{definition}\label{def:def1}
	We choose a suitable finite discretization $ U_{0}<\ldots <U_{K} $ and  take $ \infty>s_{n}(U_{K})>s_{n}(U_{K-1})>\ldots> s_{n}(U_{0})$, given some large enough constant $ K $. For a positive constant $ C $ we define the \textit{Lepski\u\i\/ principle} as  
	\begin{equation}\label{opj}
	\hat{j} = \inf\Big\{j = 0,1, \dots,K-1 |\exists k\leq j: d\left(\widehat{C}^{12}_{n, j+1}, \widehat{C}^{12}_{n, k}\right)\leq Cs_{n}(U_{j+1})\Big\}\wedge K.
	\end{equation}
	
\end{definition}
Heuristically, we want a rule so as the stochastic error will dominate the bias. We iterate the above stopping rule using the following algorithm.

{\centering
	\begin{minipage}{\linewidth}
		\begin{algorithm}[H]
			\caption{StoppingRule}\label{algo1}
			\begin{algorithmic}
				\State \textbf{Initialize:}\quad $ j:= 0$; 
				\While {$ j \leq K-1 $}
				\For {\textbf{all} $ k=0,1,\ldots,j $}
				\State Calculate $  d\left(\widehat{C}^{12}_{n, j+1} ,\widehat{C}^{12}_{n, k}\right)$
				\If { $d\left(\widehat{C}^{12}_{n, j+1} , \widehat{C}^{12}_{n, k}\right)\leq Cs_{n}(U_{j+1})$} 
				\State Accept $ j+1$, set $ \widehat{j} =j+1$
				\Else
				\State Set $\widehat{j} = K$
				\EndIf
				\EndFor
				\EndWhile
			\end{algorithmic}\label{sl}
		\end{algorithm}
	\end{minipage}
}
\vspace{0.4cm}

We observe that Lepski\u\i's strategy for parameter choice uses pairwise comparison of estimators. Comparing estimators then amounts to comparing their rates. We know that the estimator $ \widehat{C}^{12}_{n}(U_{n}) $ achieves the minimax rate for estimating the covariance. In order to find the rates, we consider the error bound which is standard in statistical learning, and corresponds to the probabilistic inequality of the form (\ref{tight}). As a result, we simplify \Cref{def:def1} by exploiting the above observation. Hence, we can substitute the right-hand side of (\ref{opj}) with the minimax rate. Let us introduce the following assumption, which summarizes the above discussion for the ``optimal" stopping index.

\begin{assumption}\label{ass2} 
	Let $ \widehat{C}^{12}_{n, j} $ be a sequence of estimators, for $ j \in \{0, 1, \dots, K\}$. We assume that
	\begin{itemize}
		\item There exists an a-priori index $ K $ such that $ \widehat{C}^{12}_{n, j} $ is well-defined for $ 0 \leq j \leq K $;
		\item There exists an optimal ``stopping'' index $ j \in \{0,1, \dots, K\}$, a large enough constant $ A $, and an increasing function $ w_{n}: \mathbb{N}_{0}\to \R^{+} $ such that
		\begin{equation}\label{opr}
		d\left( \widehat{C}^{12}_{n, i}, C^{12}\right) \leq Aw_{n}(i) \quad \mbox{for $ i \in \{j, \dots, K\} $};
		\end{equation}
		\item For each $j \in \{0, 1, \dots, K\} $, one has available a rate asymptotically minimax estimator $ \widehat{C}^{12}_{n, j} $.
	\end{itemize}
\end{assumption}
Assumption (\ref{opr}) means that the total error after the ``optimal'' stopping index $ j $ and before $ K $ is bounded by the minimax rates up to a constant $ A $.
\begin{remark}\label{trian}
	By triangle inequality and Assumption \ref{ass2}, for $ i\leq j $, we get that
	\begin{equation}\label{rem}
	d\left(\widehat{C}^{12}_{n, i}, \widehat{C}^{12}_{n, j}\right)\leq d\left(\widehat{C}^{12}_{n, i}, C^{12}\right) + d\left(\widehat{C}^{12}_{n, j}, C^{12}\right) \leq 2 Aw_{n}(j).
	\end{equation}
	
\end{remark}

As a result, we may replace the right hand side of the equation (\ref{opj}) with the rates of the estimator up to a constant $ A $. Therefore, we can define an alternative Lepski\u\i\/ principle. The following definition is inspired by \cite{birge} and (\ref{rem}).
\begin{definition}\label{def2}
	We choose a suitable finite discretization $ U_{0}<\ldots <U_{K} $, and given some large enough constant $ A $, then the Lepski\u\i\/ principle satisfies:
	\begin{equation*}
	j^{*} := \inf\Big\{j \leq K-1: \forall k\in\{j+1,\ldots,K\}  d\left(\widehat{C}^{12}_{n, j}, \widehat{C}^{12}_{n, k}\right)\leq 2Aw_{n}(k)\Big\}\wedge K,
	\end{equation*}
	where $ w_{n}(k)$ is defined as in \Cref{ass2}.
\end{definition}
Using the above definition, we know introduce the following algorithm for choosing the ``optimal'' stopping index $ j^{*} $.

{\centering
	\begin{minipage}{\linewidth}
		\begin{algorithm}[H]
			\caption{StoppingRuleOptimalRates}\label{algo2}
			\begin{algorithmic}
				\State \textbf{Initialize:}\quad $ j :=0; $ 
				\While {$ j \leq K-1 $}
				\For {\textbf{all}  $ k = j+1,\dots, K $}
				\State Calculate $  d\left(\widehat{C}^{12}_{n, j} ,\widehat{C}^{12}_{n, k}\right)$ 
				\If { $d\left(\widehat{C}^{12}_{n, j} , \widehat{C}^{12}_{n, k}\right)\geq 2Aw_{n}(k)$} 
				\State Accept $ j$, set $ j^{*} = j$
				\Else 
				\State Set $ j^{*} = K $
				\EndIf
				\EndFor
				\EndWhile
			\end{algorithmic}\label{sr}
		\end{algorithm}
	\end{minipage}
}

\vspace{0.4cm}

The next Lemma gives us an upper bound for the error estimation choosing the optimal index using \Cref{ass2}.
\begin{lemma}\label{errorthm}
	Grant Assumption \ref{ass2}, for $ j\leq K$, the error at the ``stopping'' index $ j^{*} $ in \Cref{def2} satisfies
	\begin{equation*}
	d\left(\widehat{C}^{12}_{n, j^{*}}, C^{12}\right)\leq 3A w_{n}(j).
	\end{equation*}
\end{lemma}
\begin{proof}
	Since $ w_{n}(\cdot)$ is increasing, we have 
	\begin{equation}
	d\left(\widehat{C}^{12}_{n, j}, \widehat{C}^{12}_{n, i}\right)\leq d\left(\widehat{C}^{12}_{n, i}, C^{12}\right) + d\left(\widehat{C}^{12}_{n, j}, C^{12}\right) \leq 2 Aw_{n}(j) 
	\end{equation}
	for $ j = i+1, \dots, K $. \Cref{def2} implies $ j^{*} \leq j$. Therefore,
	\begin{equation*}
	d\left(\widehat{C}^{12}_{n, j^{*}}, C^{12}\right)\leq d\left(C^{12}, \widehat{C}^{12}_{n,j}\right) + d\left(\widehat{C}^{12}_{n, j^{*}}, \widehat{C}^{12}_{n, j}\right)\leq Aw_{n}(j) + 2Aw_{n}(j) = 3 Aw_{n}(j),
	\end{equation*}
	and the proof is complete.
\end{proof}
\section{\textbf{Analysis of the stochastic error}}\label{sec:stochasticerror}
The main objective of the present section is to  prove a high probability bound for the stochastic error. Observing the form of the stochastic error $ H_{n}$ in Assumption \ref{st}, it becomes clear that we need to control the empirical characteristic function in the denominator, which may lead to unfavorable behavior for the stochastic error. To overcome this problem we consider the results obtained in \Cref{sec:characteristic}.\par

In comparison with other adaptive results obtained in \cite{comte2010nonparametric} and \cite{comte2011data}, whose procedure depends on a semiparametric assumption concerning the decay of the characteristic function, our assumption introduces a threshold to ensure that the characteristic function guards large values and the estimator makes sense.\par

\begin{lemma}\label{stochasticerror}
	Under the conditions of Lemma \ref{uniform}, the stochastic error satisfies, up to an absolute constant $ C $,
	\begin{equation}
	\E[\mathds{1}_{\mathcal{E}\cup \tilde{\mathcal{E}}}\cdot |H_{n}(U)|]\lesssim U^{-2}(n\log n)^{1/2}(w(U))^{-1}\bigg(\frac{1}{|\phi_{n}(\textbf{u})|} \vee \frac{1}{|\phi_{n}(\tilde{\textbf{u}})|}\bigg).
	\end{equation}
\end{lemma}

\begin{proof}
	
	From \Cref{errordecomposition} the stochastic error satisfies 
	
	\begin{equation}\label{stochasticdecomposition}
	\begin{aligned}
	|H_{n}(U)|&\leq\frac{n}{2U^{2}}\Bigg|\log \left|\frac{\cf(\tilde{\textbf{u}})}{\ \cf(\textbf{u})}\right|-\log\left|\frac{\phi_{n}(\tilde{\textbf{u}})}{\phi_{n}(\textbf{u})}\right|\Bigg|\\
	&= \frac{n}{2U^{2}}\Bigg|\log \left|1+\frac{\cf(\tilde{\textbf{u}})- \phi_{n}(\tilde{\textbf{u}})}{\phi_{n}(\tilde{\textbf{u}})}\right|-\log \left|1+\frac{\cf(\textbf{u})- \phi_{n}(\textbf{u})}{\phi_{n}(\textbf{u})}\right|\Bigg|.
	\end{aligned}
	\end{equation}
	On the event $ \mathcal{E} $ and $ \tilde{\mathcal{E}} $ from Lemma \ref{set}, in the case that $ |\phi_{n}(\textbf{u})|\geq \kappa_{n}n^{-1/2}  $ and $ |\phi_{n}(\tilde{\textbf{u}})|\geq \kappa_{n}n^{-1/2}  $, it yields that $ \bigg|\frac{\widehat{\phi}_{n}(\textbf{u}) - \phi_{n}(\textbf{u})}{\phi_{n}(\textbf{u})}\bigg|\leq \frac{1}{2} $ and $\bigg|\frac{\widehat{\phi}_{n}(\tilde{\textbf{u}}) - \phi_{n}(\tilde{\textbf{u}})}{\phi_{n}(\tilde{\textbf{u}})}\bigg|\leq \frac{1}{2} $. 
	The above observations lead to 
	\begin{equation}
	\begin{aligned}
	\E[\mathds{1}_{\mathcal{E}\cup \tilde{\mathcal{E}}}\cdot |H_{n}(U)|] &\leq \frac{n}{2U^{2}}\bigg(\bigg|\frac{\cf(\tilde{\textbf{u}})- \phi_{n}(\tilde{\textbf{u}})}{\phi_{n}(\tilde{\textbf{u}})} \bigg|+ \bigg| \frac{\cf(\textbf{u})- \phi_{n}(\textbf{u})}{\phi_{n}(\textbf{u})}\bigg|\bigg)\\
	&\leq \frac{Cn}{U^{2}} \bigg(\frac{\log{n}}{n}\bigg)^{1/2}(w(U))^{-1}\bigg(\frac{1}{|\phi_{n}(\textbf{u})|}\vee \frac{1}{|\phi_{n}(\tilde{\textbf{u}})|}\bigg),
	\end{aligned}
	\end{equation}
	which concludes the proof.
\end{proof}
Hence, everything boils down to controlling the unknown characteristic function in the denominator in a way that keeps the characteristic function large enough and enables a reasonable estimator. Using \Cref{emptotrue} and the inequality (\ref{inequality}), we can substitute the unknown $ \frac{1}{|\phi_{n}(\textbf{u})|} $ with $ \frac{1}{|\tilde{\phi}_{n}(\textbf{u})|} $, which is data-dependent. Inserting inequality (\ref{inequality}) into (\ref{stochasticdecomposition}), we get the following high probability upper bound for the stochastic error
\begin{equation}\label{stochastictcf}
\E[\mathds{1}_{\mathcal{E}\cup \tilde{\mathcal{E}}}\cdot|H_{n}(U)|]\leq \frac{2C}{U^{2}}\left(n \log n\right)^{1/2}(w(U))^{-1}\bigg(\frac{1}{|\tilde{\phi}_{n}(\textbf{u})|} \vee \frac{1}{|\tilde{\phi}_{n}(\tilde{\textbf{u}})|}\bigg).
\end{equation}

\begin{corollary}\label{whbstocherror}
	Under the conditions of Lemma \ref{set}, for any $ p> 1/2 $ there exists a positive constant $ C $, such that, for all $ U$, we have
	\begin{equation}
	\Pm\bigg[|H_{n}(U)|>\frac{C}{U^{2}}(n\log n)^{1/2}(w(U))^{-1}\bigg]\lesssim n^{-p}.
	\end{equation}
\end{corollary}
\begin{proof}
	The proof is a consequence of Lemma \ref{set} and \ref{stochasticerror} applying Markov inequality.
\end{proof}
\subsection{The oracle start of the parameter $ U $}
In order to apply a Lepski\u\i-type stopping rule, we need to ensure that the bound for the stochastic error is monotonically increasing. First we introduce some further notation.
\subsubsection{Further notation}
We write $ U_{start}^{oracle} $ as the staring point for the Lepski\u\i\/ principle. By (\ref{optimal}), we denote the optimal choice for the parameter, as $ U_{n} = \sqrt{\frac{r-1}{M}n\log n}$. 
We also denote as $C_{sum} = \sum_{i,j}  C_{ij}$, i.e. the sum of all elements of the covariance matrix.\par

We allow the bound for the stochastic error to depend either on the (possibly) unknown characteristic function or on the truncated empirical characteristic function. Since we can interchange w.h.p. between the true and empirical characteristic function, we use two different notations for the corresponding bounds of the stochastic error:
\begin{equation}\label{boundT}
s_{n}(U) := CU^{-2}(n\log n)^{1/2}(w(U))^{-1} \frac{1}{|\phi_{n}(\textbf{u})|} 
\end{equation}
\begin{equation}\label{boundE}
\tilde{s}_{n}(U):= CU^{-2}(n\log n)^{1/2}(w(U))^{-1}\frac{1}{|\tilde{\phi}_{n}(\textbf{u})|}.
\end{equation}
We use these bounds for the stochastic error because it easy to check that $\frac{1}{|\tilde{\phi}_{n}(\textbf{u})|} \vee \frac{1}{|\tilde{\phi}_{n}(\tilde{\textbf{u}})|} = \frac{1}{|\tilde{\phi}_{n}(\textbf{u})|} $.
In what follows, we occasionally use $ \tilde{\phi}_{n}(U) $ instead of $ \tilde{\phi}_{n}(\textbf{u}) $ because we are estimating the characteristic function on the diagonal. Same rule applies for the function $ h(\textbf{u}):=h(U,U) = 2\int_{\R^{2}}1-\cos(\langle \textbf{u}, \textbf{x}\rangle)F(dx)$.\par

Figure \ref{fig: seplot} illustrates the performance of the bound for the stochastic error using the bound $ \tilde{s}_{n}(U) $ and the stochastic error $ H_{n}(U) $, which is defined as in Assumption \ref{st}. We observe that the stochastic error is decreasing in the beginning and then it explodes. The occurrence of $ |\widehat{\phi}_{n}(\textbf{u})| $ in the denominator might have unfavorable effects.
\begin{figure}[H]
\centering
\includegraphics[width=0.85\linewidth]{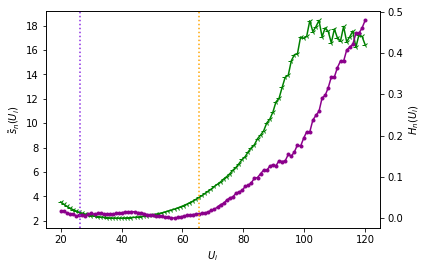}
\caption{Vertical lines: (violet-dashed) $ U_{start}^{oracle} $; (orange-dashed) $ U_{n} $. Curves: (green) $ \tilde{s}_{n}(U) $- values in the left-side $ y- $axis ; (purple) $ H_{n}(U) $- values in the right-side $ y- $ axis. }
\label{fig: seplot}
\end{figure}
To obtain a possible remedy, we consider starting the Lepski\u\i \/ procedure for a larger $ U $ and constructing a monotonically increasing bound for the stochastic error. Figure \ref{fig: seplot} depicts the above behavior.\par
We define the oracle start of $ U$ as follows:
\begin{equation}\label{oracleU}
U^{oracle}_{start} = \inf \bigg\{U>0 : |\phi_{n}(U)| \leq \frac{1}{2}\bigg\}.
\end{equation}
Let us highlight the strategy of constructing a monotonically increasing bound for the stochastic error. Finding the oracle start of $ U $, we show that $ U_{start}^{oracle}<U_{n} $.  Then, we prove that $ s_{n}(U_{start}^{oracle}) < s_{n}(U_{n})$, ensuring that an increasing bound is available for the stochastic error, within the interval $ [U_{start}^{oracle}, U_{end}] $ for $ U_{end}>U_{n} $. The above discussion is depicted in Figure \ref{fig: seplot}. It is worth emphasizing that the calculation of $ U_{start}^{oracle} $ requires the evaluation of the perhaps unknown $ \phi_{n}(\textbf{u}) $. Thus, we take into consideration only a general assumption for the characteristic function, like the quasi-monotonicity of Assumption \ref{quasimonoton} for infinity variation co-jumps, i.e., $ r\in (1,2] $ and a boundedness condition for the covariance matrix.

\begin{lemma}\label{uoracle}
	For big $ n $, and  $ K>0 $, the interval for $ U_{start}^{oracle} $ is
	\begin{equation*}
	\bigg[	\frac{\sqrt{2\log 2}}{\sqrt{C_{sum} +K}}\cdot \sqrt{n}, \quad\frac{\sqrt{2\log 2}}{\sqrt{C_{sum}}}\cdot \sqrt{n}\bigg]
	\end{equation*}
	and for $ r\in(1,2] $, we get that 
	\[ 
	U_{start}^{oracle}< U_{n},
	\]
	where $ U_{n} = \sqrt{\frac{r-1}{M}n\log n} $.
\end{lemma}
\begin{proof}
	The absolute value of the characteristic function is given by 
	\begin{equation}\label{cf}
	|\phi_{n}(\textbf{u})| = \exp \bigg\{-\frac{1}{2n}\bigg(\langle C\textbf{u}, \textbf{u}\rangle + h(\textbf{u})\bigg)\bigg\} ,
	\end{equation}
	where $ \textbf{u} = (U, U) $. We define 
	\[
	h(\textbf{u}) = 2\int_{\R^{2}}1 -\cos(\langle\textbf{u}, \textbf{x}\rangle)F(d\textbf{x}),
	\]
	where $ F$ is the L\'evy measure in $ \R^{2} $. Using the Cauchy-Schwarz inequality for $ | \langle \textbf{u}, \textbf{x}\rangle|^{2}\leq \|\textbf{u}\|^{2}\|\textbf{x}\|^{2}$, a positive constant $ K $ and $ v_{0} = (0,1)^{2} \subset \R^{2}$
	\begin{equation}\label{smalljumps}
	\begin{aligned}
	h(\textbf{u}) = 2 \int_{\R^{2}}\Big(1 - \cos( \langle \textbf{u}, \textbf{x}\rangle ) \Big)F(d\textbf{x}) & =  2 \int_{v_{0}} \Big(1 - \cos(\langle \textbf{u}, \textbf{x}\rangle )\Big)F(d\textbf{x})\\
	& + 2\int_{\R^{2}\setminus v_{0}} \Big(1 -\cos(\langle \textbf{u}, \textbf{x}\rangle\Big) F(d\textbf{x})\\
	&\leq 2\int_{v_{0}}|\langle\textbf{u}, \textbf{x} \rangle|^{2}F(d\textbf{x}) +4\int_{\R^{2}\setminus v_{0}}dF(\textbf{x})\\
	&\leq 4U^{2}\int_{v_{0}}\|x\|^{2} dF(\textbf{x})  +4 F(\R^{2}\setminus v_{0})\\
	&\leq KU^{2}
	\end{aligned}
	\end{equation}
	The last inequality derives from the fact that we always have $ \int_{\R^{2}} (1 \wedge \|\textbf{x} \|^{2} ) F(d\textbf{x}) < \infty$. So we can obtain the following inequality
	\begin{equation}\label{sjineq}
	0\leq h(\textbf{u})\leq K U^{2}.
	\end{equation}
	It is easy to check that $ \langle C\textbf{u}, \textbf{u}\rangle = C_{sum}U^{2} $. Inserting this fact and (\ref{sjineq}) into (\ref{cf}) we get the following inequality for the absolute value of the characteristic function
	\begin{equation*}
	\exp\bigg\{-\frac{(C_{sum}+K)U^{2}}{2n}\bigg\} \leq |\phi_{n}(\textbf{u})|\leq \exp\bigg\{-\frac{C_{sum}U^{2}}{2n}\bigg\}
	\end{equation*}
	Inserting the above inequality into (\ref{oracleU}) we get the required interval for $U_{start}^{oracle}$, which ensures that $ U_{start}^{oracle}\sim \sqrt{n} $. This implies that $ U_{start}^{oracle} <U_{n}$ for big $ n $. This concludes the proof.
\end{proof}

\begin{lemma}\label{stomonoton}
	For $ U_{start}^{oracle}< U_{n} $, the stochastic error satisfies
	\[
	s_{n}(U_{start}^{oracle})\leq s_{n}(U_{n}) .
	\]
\end{lemma}

\begin{proof}
	It suffices to show that 
	\begin{equation}\label{equation}
	\frac{s_{n}(U_{start}^{oracle}}{s_{n}(U_{n})}\leq 1.
	\end{equation}
	By the form of $ s_{n}(U) $ in (\ref{boundT}), it is easy to check that
	\begin{equation}\label{stoeq}
	\frac{s_{n}(U_{start}^{oracle}}{s_{n}(U_{n})} = \bigg(\frac{U_{n}}{U_{start}^{oracle}}\bigg)^{2}\frac{w(U_{n})}{w(U_{start}^{oracle})}\bigg|\frac{\phi_{n}(U_{n})}{\phi_{n}(U_{start}^{oracle})}\bigg|.
	\end{equation}
	By \Cref{uoracle}, it yields $ \bigg(\frac{U_{n}}{U_{start}^{oracle}} \bigg)^{2}\leq 1$. By \Cref{weight}, we also know that $ w(U) $ is a decreasing function, which means that $ \frac{w(U_{n})}{w(U_{start}^{oracle})}>1 $. 
	For the third term of (\ref{stoeq}) we have
	\begin{equation}\label{cfineq}
	\begin{aligned}
	\bigg|\frac{\phi_{n}(U_{n})}{\phi_{n}(U_{start}^{oracle})}\bigg| = \exp\bigg\{&\frac{1}{2n}\big(C_{sum}\left((U_{start}^{oracle})^{2}- U_{n}^{2}\right)\big)\\
	&+\frac{1}{2n}\big(C_{sum}\left(h(U_{start}^{oracle})- h(U_{n})\right)\big)\bigg\}.\\
	\end{aligned}
	\end{equation}
	By (\ref{sjineq}) we have that $ h(U_{start}^{oracle}) -h(U_{n})\leq h(U_{start}^{oracle}) $. We also get 
	\[ (U_{start}^{oracle})^{2} - U_{n}^{2} \leq n\bigg(\frac{2\log 2}{C_{sum}K} - \frac{r-1}{M}\log n\bigg).
	\]
	Substituting the above inequalities into (\ref{cfineq}) we obtain
	\begin{equation}\label{cfineq2}
	\begin{aligned}
	\bigg|\frac{\phi_{n}(U_{n})}{\phi_{n}(U_{start}^{oracle})}\bigg|&\leq \exp \bigg\{\frac{C_{sum} \log 2}{C_{sum}+ K}-C_{sum}\frac{r-1}{2M}\log n + \frac{K\log 2}{C_{sum} +K}\bigg\}\\
	& = \exp\bigg\{\log 2 - C_{sum}\frac{r-1}{2M}\log n\bigg\} \\
	&= \frac{2}{n^{\frac{C_{sum(r-1)}}{2M}}}.
	\end{aligned}
	\end{equation}
	Taking everything into consideration we get 
	\begin{equation*}
	\frac{s_{n}(U_{start}^{oracle})}{s_{n}(U_{n})} \leq C_{0} \frac{\log n}{n^{\frac{C_{sum(r-1)}}{2M}}}
	\end{equation*}
	which is smaller than one as $ n\to \infty $. The statement is proved.
\end{proof}

A side product of the above analysis is the following corollary, which ensures that the upper bound of the stochastic error is always monotonically increasing over the desired interval.
\begin{corollary} 
	If we set
	\begin{equation}\label{supse}
	s_{n}^{*}(u) := \sup_{U_{start}^{oracle}\leq v \leq u}s_{n}(v),
	\end{equation}
	then $ s_{n}^{*}$ satisfies 
	\begin{equation*}
	s_{n}(U_{n}) = s_{n}^{*}(U_{n}).
	\end{equation*}
\end{corollary}
\begin{proof}
	By \Cref{uoracle} and \Cref{stomonoton}, the proof is straightforward.
\end{proof}
Despite the fact that we used the (possible) unknown theoretical characteristic function as a criterion for the oracle start of Lepski\u\i\/ procedure and construct a monotonically increasing bound as we wish, it is useful to secure a data-driven criterion as well. For this reason we propose the following definition.
\begin{definition}\label{uoracledata}
	For $ c \in (0,1] $, we define the criterion for the oracle start of $ U $ as following 
	\begin{equation*}
	\widehat{U}_{start}^{oracle} := \inf \big\{U >0: |\widehat{\phi}_{n}(U)|\leq c \big\}.
	\end{equation*}
\end{definition}
The last ingredient which remains to be proven is the following high probability bound, which will allow us to connect a data-driven choice for the oracle start of the Lepski\u\i\/ procedure with the theoretical characteristic function.
\begin{lemma}\label{whpstochasticerror}
	For $ c \in(0,1] $, choosing $ \widehat{U}_{start}^{oracle} $ as in $ \Cref{uoracledata} $, there is a high probability event $ \{|\widehat{\phi}_{n}(\widehat{U}_{start}^{oracle})|\leq c\} $ satisfying 
	\begin{equation*}
	\lim_{n\to \infty} \Pm \big[|\widehat{\phi}_{n}\big(\widehat{U}_{start}^{oracle}\big)|\leq c\big] = 1,
	\end{equation*}	
 with probability at least $ 1-\exp(-2(c+1)^{2}n) $.
\end{lemma}

\begin{proof}
	The boundedness condition of the Hoeffding's inequality has to be verified for the random variables $ |e^{i\langle \textbf{u}, Y_{j} \rangle}| $, where $ Y_{j} $ are L\'evy increments. This can be done along the line of the proof of Lemma \ref{countable}. Hence,
	$ \Pm[|\widehat{\phi}_{n}(U)|\leq c] = 1 - \Pm[|\widehat{\phi}_{n}(U)|>c] $. Applying Hoeffding's inequality, we obtain
	\begin{equation*}
	\begin{aligned}
	\Pm[|\widehat{\phi}_{n}(U)|>c]& \leq \Pm \bigg[\frac{1}{n}\sum_{j=1}^{n}|e^{i \langle \textbf{u}, Y_{j}\rangle}| - \E|e^{i\langle \textbf{u}, Y_{j}\rangle}|>c\bigg]\\
	&\leq\exp(-2n(c+1)^{2}).
	\end{aligned}
	\end{equation*}
	Inserting Definition \ref{uoracledata} to the empirical characteristic function, the statement is proven.
\end{proof}
\section{\textbf{Balancing principle when the stochastic error is known}}\label{sec:bp}
In this section, we prove an upper bound for the best possible adaptive parameter using a balancing principle inspired by the work of \cite{de2010adaptive} for adaptive kernel methods. The optimal choice $ U_{n} $ crucially depends on the unknown parameters $ r, M $. Here, we construct a completely data-driven estimation procedure adapted to $ U \in \mathcal{U}$, where $ \mathcal{U} = [U_{start}^{oracle}, U_{end}] $. Our main result for the adaptive estimation shows that the Lepski\u\i\/ estimator achieves almost the optimal rates.\par

In the following we denote by $ a(n) := \exp(-2n(c+1)^{2}) $. By (\ref{boundE}), w.h.p. at least $ 1-a(n) $ the upper bound for the stochastic error will be of the form
\begin{equation}\label{sto}
\tilde{s}_{n}(U) = \frac{2C\gamma(n)}{\theta(U)}\frac{(w(U))^{-1}}{|\tilde{\phi}_{n}(U)|}.
\end{equation}
where $ \theta(U) = U^{2} $, $ \gamma(n) = (n\log n)^{1/2} $. Further, the term $ d(U) $ is the deterministic error bound, which does not depend on data and is of the form
\begin{equation}\label{det}
d(U) = \frac{M2^{r/2}}{U^{2-r}},
\end{equation}
where $ r \in (1,2] $ is the co-jump activity index and $ M $ is from the \Cref{class}. Consequently, The estimation error bound is given by the sum of two competing terms with probability at least $ 1- \exp(-2n(c+1)^{2}) $ i.e., 
\begin{equation}\label{estimationbound}
|\widehat{C}^{12}_{n}(U) - C^{12}|\leq \tilde{s}_{n}(U)+ d(U). 
\end{equation}
The upper bound of (\ref{estimationbound}) is the sum of a bias term which decreases in $ U $ and a stochastic error which increases in $ U $, for $ U \in \mathcal{U} $. According to the balancing principle, the best possible adaptive parameter choice is found by solving the bias-variance-type decomposition (\ref{estimationbound}), which implies that we have to balance the deterministic and the stochastic error. We consider that $ U_{bal} $ makes the contribution of two terms equal, i.e. $ d(U_{bal}) = \tilde{s}_{n}(U_{bal}) $. We observe that the corresponding error estimate is, with probability at least $ 1 -a(n) $,
\begin{equation}\label{best}
|\widehat{C}^{12}_{n,j}- C^{12}|\leq  2d(U_{bal}) = 2\tilde{s}_{n}(U_{bal}),
\end{equation}
where $ 0 < a(n) < 1 $ and $ U_{bal} $ is the best possible parameter. Using the equation in (\ref{best}) and the knowing the form of the stochastic error and the deterministic error by (\ref{sto}) and (\ref{det}), the theoretical best parameter according to balancing principle will satisfy
\begin{equation}\label{ubal}
U_{bal} = \bigg(\frac{4C}{\kappa 2^{r/2}M}\bigg)^{1/r}n^{1/r}.
\end{equation}
Let us now highlight the idea behind the balancing principle. It is clear, by the monotonicity of the stochastic and deterministic error, that
\begin{equation*}
\tilde{s}_{n}(U_{bal})+ d(U_{bal})\leq 2\min_{U}\{\tilde{s}_{n}(U)+d(U_{bal})\}.
\end{equation*}
If we choose $ U_{*}\leq U_{bal} $:
\begin{equation*}
\tilde{s}_{n}(U_{bal}) + d(U_{bal})\leq 2 d(U_{bal})\leq 2d(U_{*})\leq 2\min_{U}\{\tilde{s}_{n}(U)+d(U)\}.
\end{equation*}
On the other hand, if we choose $ U_{*}\geq U_{bal} $:
\begin{equation*}
\tilde{s}_{n}(U_{bal}) + d(U_{bal})\leq 2 \tilde{s}_{n}(U_{bal})\leq 2\tilde{s}_{n}(U_{*})\leq 2\min_{U}\{\tilde{s}_{n}(U)+d(U)\}.
\end{equation*}
Since $ C>0 $, we can rewrite the decomposition error bound as
\begin{equation}\label{bound}
|\widehat{C}^{12}_{n,j} - C^{12}| \leq 2 C\bigg( \frac{\gamma(n)}{\theta(U)}\frac{(w(U))^{-1}}{|\tilde{\phi}_{n}(U)|}+ d(U)\bigg),
\end{equation}
where $ d(\cdot)$ is assumed to be a continuous, monotonically non-increasing function. As a result, the corresponding best parameter choice $ U_{bal} $ gives, with probability $ 1 - a(n) $, the rate 
\begin{equation}
|\widehat{C}^{12}_{n, bal} - C^{12}|\leq 2 \frac{2C\gamma(n)}{\theta(U_{bal})}\frac{(w(U_{bal}))^{-1}}{{|\tilde{\phi}_{n}(U_{bal})|}}= 2d\left(U_{bal}\right).
\end{equation}
The aim is to choose $ U_{bal} $ from the set:
\begin{equation}
\mathcal{U} := [U_{start}^{oracle}, U_{end}] \quad \text{with} \quad U_{start}^{oracle} \sim \sqrt{n}.
\end{equation}
To define a parameter strategy, we first consider a discretization for the possible values of $U_{j} $, that is, an ordered sequence $ (U_{j})_{j\in \mathbb{N}} $ such that the best value $ U_{bal} $ falls within the considered grid $ \mathcal{U} $. The \textit{balancing principle} estimate for $ U_{bal}$ is defined via

\begin{equation}\label{uhat}
U_{\widehat{j}} = \max \Bigg\{U_{i}: \forall U_{j}\leq U_{i}, U_{i}\in \mathcal{U}, \big|\widehat{C}^{12}_{n,j} - \widehat{C}^{12}_{n,i}\big|\leq\frac{8C \gamma(n)}{\theta(U_{j})}\frac{w(U_{j})}{|\tilde{\phi}_{n}(U_{j})|} \Bigg\}.
\end{equation}
The reasons why we expect this estimate to be sufficiently close to $ U_{bal} $ and why this estimate does not depend on $ d $ are better explained with the following argument.
Observe that if we choose two indices $ \alpha, \beta $ such that $ U_{\alpha}\geq U_{\beta} \geq U_{bal} $, then with probability at least $ 1- a(n)$,
\begin{equation}\label{ba}
\begin{aligned}
|\widehat{C}^{12}_{n,\alpha} - \widehat{C}^{12}_{n,\beta}| &\leq |\widehat{C}^{12}_{n,\alpha} - C^{12}| + |\widehat{C}^{12}_{n,\beta}- C^{12}|\\
& \leq 2C\Bigg(\frac{\gamma(n)}{\theta(U_{\alpha})}\frac{(w(U_{\alpha}))^{-1}}{|\tilde{\phi}_{n}(U_{\alpha})|}+d(U_{\alpha})\Bigg) \\
&+ 2C \Bigg( \frac{\gamma(n)}{\theta(U_{\beta})}\frac{(w(U_{\beta}))^{-1}}{|\tilde{\phi}_{n}(U_{\beta})|}+d(U_{\beta})\Bigg)\\
& \leq \frac{8C\gamma(n)}{\theta(U_{\alpha})}\frac{(w(U))^{-1}}{|\tilde{\phi}_{n}(U_{\alpha})|}.
\end{aligned}
\end{equation} 
The intuition is that when such a condition is violated, we are close to the value which contributes equal to the deterministic and stochastic errors, which is $ U_{bal} $. 
\begin{prop}\label{as2}
	For a positive constant $ C$, the estimation error bound, with high probability, is written like
	\begin{equation}
	|\widehat{C}^{12}_{n}(U) -C^{12}|\leq  2 C \bigg( \frac{\gamma(n)}{\theta(U)}\frac{(w(U))^{-1}}{|\tilde{\phi}_{n}(U)|}+ d(U)\bigg),
	\end{equation}
	where 
	\begin{itemize}
		\item 
		$ d(U) $ is a continuous, non-increasing function,
		\item 
		$ \theta(U) = U^{2} $,
		\item 
		$ w(U) $ is a non-increasing function such $ 0<w(U)\leq 1 $,
		\item 
		$ |\tilde{\phi}_{n}(U)|$ is a continuous, non-increasing function for $ U\in \mathcal{U} $.
	\end{itemize}
\end{prop}
\begin{proof}
\Cref{stochasticdecomposition} gives us an upper bound for the stochastic error $ H_{n}(U)$. Similar $ d(U) $ is an upper bound for the deterministic error $ D(U) $. By \Cref{errordecomposition}, we can decompose the estimation error as a sum of stochastic error $ H_{n}(U) $ and a deterministic error $ D(U) $. Consequently, it yields, with high probability, 
\begin{equation*}
|\widehat{C}^{12}_{n}(U)-C^{12}| \leq |H_{n}(U) + D(U)|\leq s_{n}(U) + d(U).
\end{equation*}
Recall that the inequality (\ref{inequality}) allows us to interchange with high probability between the (perhaps) unknown characteristic function and the empirical characteristic function. A direct consequence of the above observation is that we can interchange with high probability between the population bound $ \tilde{s}_{n}(U) $ and the theoretical bound $ s_{n}(U) $ for the stochastic error
\begin{equation}\label{stochange}
\frac{1}{2} s_{n}(U)\leq \tilde{s}_{n}(U)\leq 3 s_{n}(U).
\end{equation}
This leads to $ s_{n} + d(U) \leq 2 \tilde{s}_{n}(U) + d(U)$. Now taking into consideration the form of $ \tilde{s}_{n}(U) $ in (\ref{sto}), the proof is complete.
\end{proof}
\Cref{as2} requires the bound to be uniform with respect to $ U_{j} $, since the parameter choice is data-dependent. The uniform condition is satisfied due to the Lemma \ref{uniform}.
Driven by the inequality (\ref{stochange}), the strategy for the balancing principle will give us with high probability
\begin{equation}\label{stobalance}
\begin{aligned}
\tilde{s}_{n}(U_{bal}) &\leq 2 \min_{U}\{\tilde{s}_{n}(U)+d(U)\}\\
&\leq 2 \min_{U}\{3 s_{n}(U) +d(U)\}\\
&\leq 6 \min_{U}\{s_{n}(U)+d(U)\}\\
&\leq 6 s_{n}(U_{bal}).
\end{aligned}
\end{equation}
The following theorem shows that the value $ U_{\widehat{j}} $, given by the balancing principle (\ref{uhat}), provides the same estimation error of $ U_{bal} $ up to a constant. Note that all the inequalities in the following proofs are to be interpreted as holding with high probability. We can now prove the convergence rate for the adaptive estimator $ \widehat{C}^{12}_{n, \widehat{j}} $.  
\begin{theorem}\label{result1}
	Using \Cref{as2} for a sequence of parameters $ U_{j} $  which satisfies $ U_{j} \in [U_{start}^{oracle}, U_{end}] $,
	then there is a constant $ C $ such that the adaptive estimator satisfies with high probability 
	\begin{equation*}
	\Pm\big[|\widehat{C}^{12}_{n, \widehat{j}} - C^{12}|\leq 5\tilde{s}_{n}(U_{bal})\big]\geq 1- a(n).
	\end{equation*}
\end{theorem}
\begin{proof}
	Recall that by (\ref{ba}) for $ \alpha, \beta $ such that $ U_{\alpha} \geq U_{\beta}\geq U_{bal}$ we have
	\begin{equation*}
	|\widehat{C}^{12}_{n,\alpha} -\widehat{C}_{n,\beta}^{12}|\leq \frac{8C\gamma(n)}{\theta(U_{\alpha})}\frac{(w(U_{\alpha}))^{-1}}{|\tilde{\phi}_{n}(U_{\alpha})|}.
	\end{equation*}
	It is easy to see that $ U_{bal}\leq U_{\widehat{j}}$. From the definition of $ U_{\widehat{j}} $, $ U_{bal} $ and the triangle inequality, it yields
	\begin{equation}\label{eq}
	\begin{aligned}
	|\widehat{C}^{12}_{n, \widehat{j}} - C^{12}|&\leq |\widehat{C}^{12}_{n, \widehat{j}} - \widehat{C}^{12}_{n, bal}| + |\widehat{C}^{12}_{n,bal} - C^{12}|\\
	&\leq \frac{8C\gamma(n)}{\theta(U_{bal})}\frac{(w(U_{bal}))^{-1}}{|\tilde{\phi}_{n}(U_{bal})|} + 2C\bigg(d(U_{bal})+ \frac{\gamma(n)}{\theta(U_{bal})}\frac{(w(U_{bal}))^{-1}}{|\tilde{\phi}_{n}(U_{bal})|}\bigg)\\
	&\leq 5 \tilde{s}_{n}(U_{bal}),
	\end{aligned}
	\end{equation}
	as required.
	\end{proof}
	\section{Discussion}\label{sec:discussion}
\textit{Comments on the stochastic error.} The bound (\ref{stochastictcf}) ensures that the stochastic error of the estimator is upper-bounded by the truncated characteristic function up to a logarithmic factor and a multiplicative constant $ C $. This means that the bound depends on the random quantity $ |\tilde{\phi}_{n}(\textbf{u})| $ but not on the unknown characteristic function $ \phi_{n}(\textbf{u}) $. With high probability, the inequality (\ref{inequality}) allows us to interchange between the unknown characteristic function and the truncated empirical characteristic function, which is data-dependent. A direct consequence of the above observation is that we can interchange with high probability between the population bound $ \tilde{s}_{n}(U) $ and the theoretical bound $ s_{n}(U) $ for the stochastic error. \par
In Figure \ref{fig: seplot}, an irregular behavior of the bound for the stochastic error is observed for small values of $ U $, because of the empirical characteristic function in the denominator. To overcome this obstacle, we find an oracle start for $ U $, so as to ensure a monotonically increasing bound for the stochastic error.\par  
\textit{Comments on the balancing principle.} The construction of a monotonically increasing bound for the stochastic error allows us to apply Lepski\u\i\/'s\/ principle for the adaptive estimator $ \widehat{C}^{12}_{n,\hat{j}} $. Theorem \ref{result1} shows us that the balancing principle can adaptively achieve the best possible rate, which is near-optimal in a minimax sense.

\section{Numerical experiments}\label{sec:experiments}
In this section we test the behavior of the covariance estimator in order to adapt the parameter $ U $, i.e. the frequency for estimating the covariance. This means that we first have to simulate a bivariate L\'evy process on $ [0,1] $. We will draw our observations from a process $ X_t = B_t + J_t$, where $ X_t $ is a superposition of a two-dimensional Brownian motion $ B_{t} $and a two-dimensional jump process $ J_{t} $. Its jumps are driven by a two-dimensional $ r_{i} $-stable process for $ i = 1,2 $ where $ r_{i}\in (0,2] $. $ X_t $ thus models a process with both diffusion and jump components. We assume the covariance matrix has the form $ C = \begin{psmallmatrix}
2 & 1\\ 
1  & 1
\end{psmallmatrix} $. In each run of our simulation, we will generate $ n = 1,000 $ observations, corresponding to observations taken every $ 1/1,000 $ over a time interval $ [0, 1] $ and $ U_{i} \in [0.1, 50]$ for $ i \in\{1, 2, \dots, 500\} $.\par 

We conduct several experiments for $ U $, using different choices for the jump index activity. We start with jumps of finite variation, i.e. $ r_{i}\in[0.1, 0.9] $, then we continue with jumps of infinite variation, i.e. $ r_{i} \in [1.1, 1.8] $.  In the following experiments, we plot (\ref{fig:fv1} - \ref{fig:fv12}) the empirical characteristic function $\widehat{\phi}_{n}(U_{i})  $, the real and positive part of empirical characteristic function $ \widehat{\phi}_{n}(\tilde{U}_{i}) $, $ \log|\phi_{n}(U_{i})| $, $ \log|\phi_{n}(\tilde{U}_{i})| $ and $  \log|\phi_{n}(\tilde{U}_{i})|- \log|\phi_{n}(U_{i})|$ against the parameter for adaptation $ U_{i}$. \par

Plots \ref{fig:cov1}, \ref{fig:cov2}, \ref{fig:cov3}, \ref{fig:cov4} show that the estimator is consistent to the true  when $ U_{i} $ ranges from around $ 5 $ till $ 30 $.  Recall that $ C^{12} = 1 $. As we expected the behavior of estimator in the beginning and at the end of the interval is quite erratic, because the bias of the estimator is quite high. In principle, for $ U_{i} = 30$ we have the ``optimal'' stopping index.

\begin{figure}[H]
	\begin{subfigure}[b]{0.5\textwidth}
		\centering
		\includegraphics[width=0.85\linewidth]{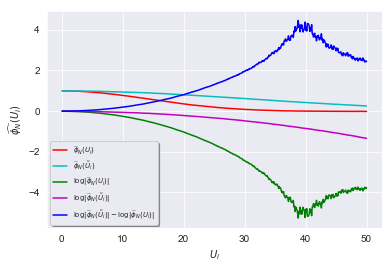}
		\caption{$ r_{1} = 0.2, r_{2} = 0.1$}
		\label{fig:fv1}
	\end{subfigure}%
	\begin{subfigure}[b]{0.5\textwidth}
		\centering
		\includegraphics[width=.85\linewidth]{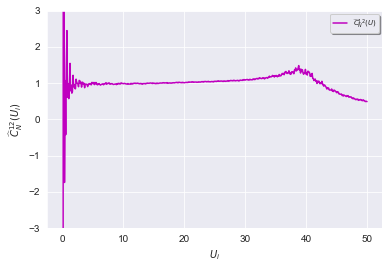}
		\caption{Covariance estimator}
		\label{fig:cov1}
	\end{subfigure}%
\end{figure}

\begin{figure}[H]
	\begin{subfigure}[b]{0.5\textwidth}
		\centering
		\includegraphics[width=0.85\linewidth]{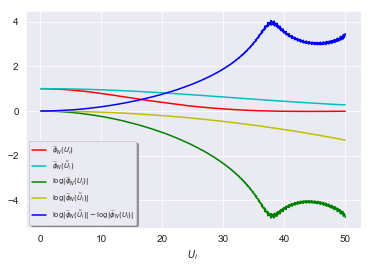}
		\caption{$r_{1} = 0.4, r_{2}= 0.3$}
		\label{fig:fv2}
	\end{subfigure}%
	\begin{subfigure}[b]{0.5\textwidth}
		\centering
		\includegraphics[width=0.85\linewidth]{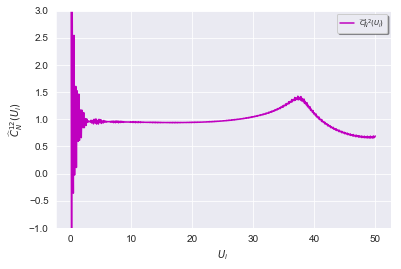}
		\caption{Covariance estimator}
		\label{fig:cov2}
	\end{subfigure}%
\end{figure}

\begin{figure}[H]
	\begin{subfigure}[b]{0.5\textwidth}
		\centering
		\includegraphics[width=0.85\linewidth]{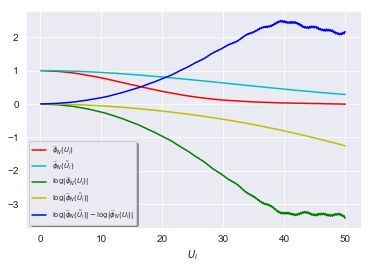}
		\caption{$ r_{1} = 0.5,  r_{2} = 0.4 $}
		\label{fig:fv3}
	\end{subfigure}%
	\begin{subfigure}[b]{0.5\textwidth}
		\centering
		\includegraphics[width=0.85\linewidth]{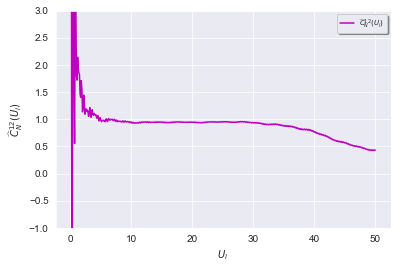}
		\caption{Covariance estimator}
		\label{fig:cov3}
	\end{subfigure}%
\end{figure}

\begin{figure}[H]
	\begin{subfigure}[b]{0.5\textwidth}
		\centering
		\includegraphics[width=0.85\linewidth]{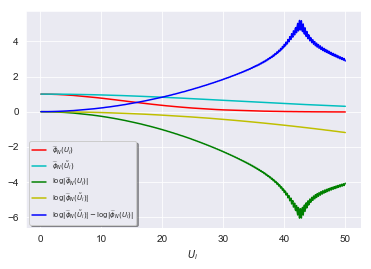}
		\caption{$ r_{1} = 0.6,  r_{2} = 0.5 $}
		\label{fig:fv4}
	\end{subfigure}%
	\begin{subfigure}[b]{0.5\textwidth}
		\centering
		\includegraphics[width=0.85\linewidth]{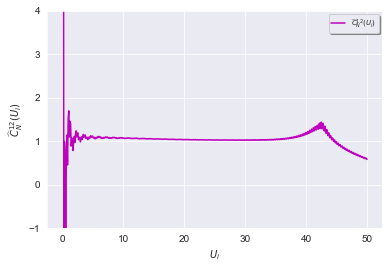}
		\caption{Covariance estimator}
		\label{fig:cov4}
	\end{subfigure}%
	\label{fig:FV2}
\end{figure}

\begin{figure}[H]
	\begin{subfigure}[b]{0.5\textwidth}
		\centering
		\includegraphics[width=0.85\linewidth]{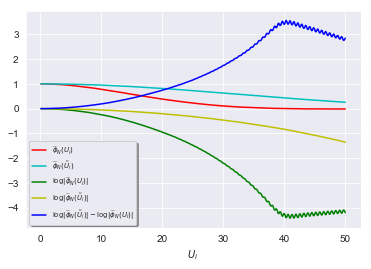}
		\caption{$ r_{1} = 0.7,  r_{2} = 0.6$}
		\label{fig:fv5}
	\end{subfigure}%
	\begin{subfigure}[b]{0.5\textwidth}
		\centering
		\includegraphics[width=0.85\linewidth]{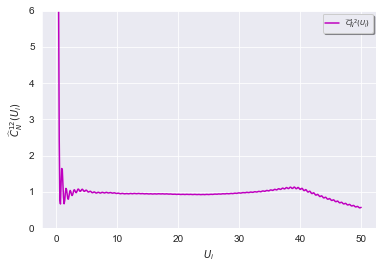}
		\caption{Covariance estimator}
		\label{fig:cov5}
	\end{subfigure}%
	\label{fig:FV3}
\end{figure}

\begin{figure}[H]
	\begin{subfigure}[b]{0.5\textwidth}
		\centering
		\includegraphics[width=0.85\linewidth]{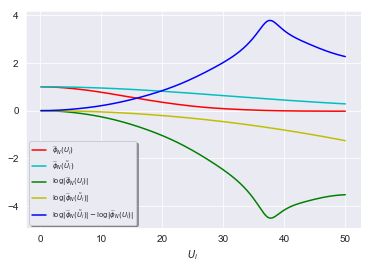}
		\caption{$ r_{1} = 0.8,  r_{2} = 0.7 $}
		\label{fig:fv6}
	\end{subfigure}%
	\begin{subfigure}[b]{0.5\textwidth}
		\centering
		\includegraphics[width=0.85\linewidth]{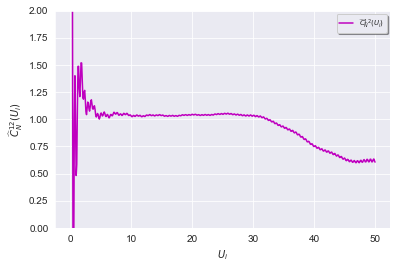}
		\caption{Covariance estimator}
		\label{fig:cov6}
	\end{subfigure}
	\label{fig:FV4}
\end{figure}

\begin{figure}[H]
	\begin{subfigure}[b]{0.5\textwidth}
		\centering
		\includegraphics[width=0.85\linewidth]{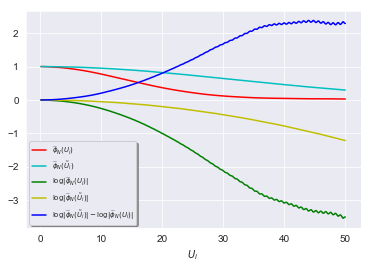}
		\caption{$ r_{1} = 0.9,  r_{2} = 0.8 $}
		\label{fig:fv7}
	\end{subfigure}%
	\begin{subfigure}[b]{0.5\textwidth}
		\centering
		\includegraphics[width=0.85\linewidth]{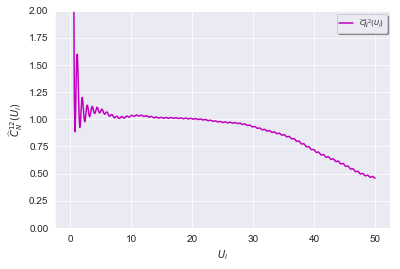}
		\caption{Covariance estimator}
		\label{fig:cov7}
	\end{subfigure}%
	\label{fig:FV5}
\end{figure}

\begin{figure}[H]
	\begin{subfigure}[b]{0.5\textwidth}
		\centering
		\includegraphics[width=0.85\linewidth]{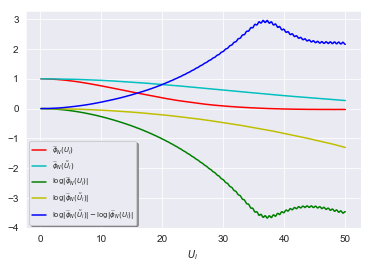}
		\caption{$ r_{1} = 1.0,  r_{2} = 0.9 $}
		\label{fig:fv8}
	\end{subfigure}%
	\begin{subfigure}[b]{0.5\textwidth}
		\centering
		\includegraphics[width=0.85\linewidth]{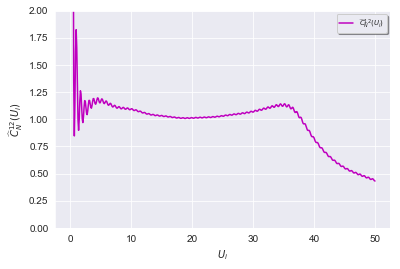}
		\caption{Covariance estimator}
		\label{fig:cov8}
	\end{subfigure}%
	\label{fig:FV6}
\end{figure}

Next, we plot bivariate L\'evy processes with at least one jump component of infinite variation. Plots \ref{fig:cov8}, \ref{fig:cov9}, \ref{fig:cov10}, \ref{fig:cov11}, \ref{fig:cov12} show the behavior of estimator is not consistent with the theoretical one. Henceforth, the Lepskii's method can not be applied, especially in the case of \ref{fig:fv8} and \ref{fig:fv12}. 

\begin{figure}[H]
	\begin{subfigure}[b]{0.5\textwidth}
		\centering
		\includegraphics[width=0.85\linewidth]{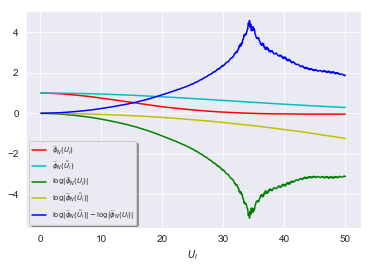}
		\caption{$ r_{1} = 1.1,  r_{2} = 1.0 $}
		\label{fig:fv9}
	\end{subfigure}%
	\begin{subfigure}[b]{0.5\textwidth}
		\centering
		\includegraphics[width=0.85\linewidth]{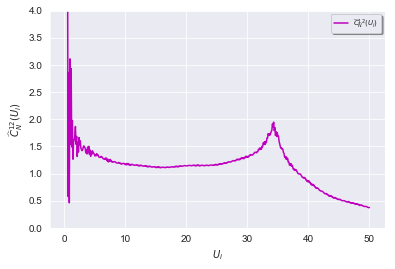}
		\caption{Covariance estimator}
		\label{fig:cov9}
	\end{subfigure}%
	\label{fig:FV7}
\end{figure}

\begin{figure}[H]
	\begin{subfigure}[b]{0.5\textwidth}
		\centering
		\includegraphics[width=0.85\linewidth]{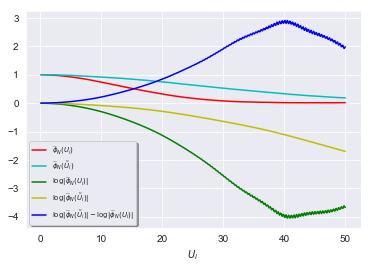}
		\caption{$ r_{1} = 1.5,  r_{2} = 0.5 $}
		\label{fig:fv10}
	\end{subfigure}%
	\begin{subfigure}[b]{0.5\textwidth}
		\centering
		\includegraphics[width=0.85\linewidth]{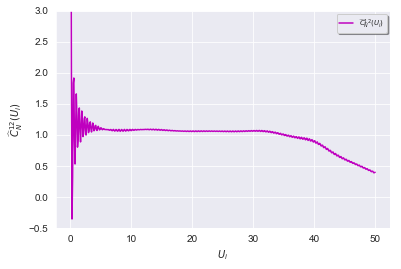}
		\caption{Covariance estimator}
		\label{fig:cov10}
	\end{subfigure}%
	\label{fig:FV8} 
\end{figure}

\begin{figure}[H]
	\begin{subfigure}[b]{0.5\textwidth}
		\centering
		\includegraphics[width=0.85\linewidth]{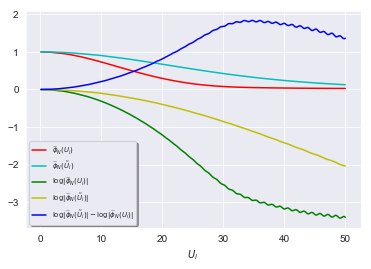}
		\caption{$ r_{1} = 1.8,  r_{2} = 0.5 $}
		\label{fig:fv11}
	\end{subfigure}%
	\begin{subfigure}[b]{0.5\textwidth}
		\centering
		\includegraphics[width=0.85\linewidth]{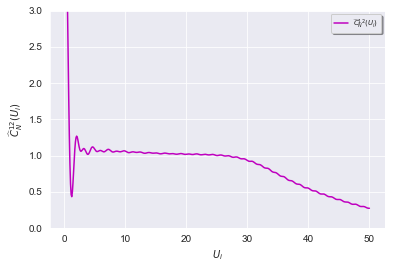}
		\caption{Covariance estimator}
		\label{fig:cov11}
	\end{subfigure}%
	\label{fig:FV9}
\end{figure}

\begin{figure}[H]
	\begin{subfigure}[b]{0.5\textwidth}
		\centering
		\includegraphics[width=0.85\linewidth]{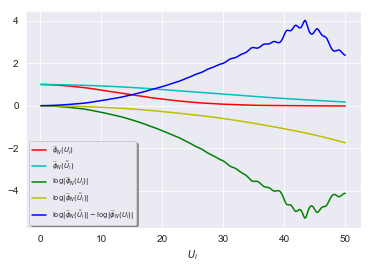}
		\caption{$ r_{1} = 1.5,  r_{2} = 1.0 $}
		\label{fig:fv12}
	\end{subfigure}%
	\begin{subfigure}[b]{0.5\textwidth}
		\centering
		\includegraphics[width=0.85\linewidth]{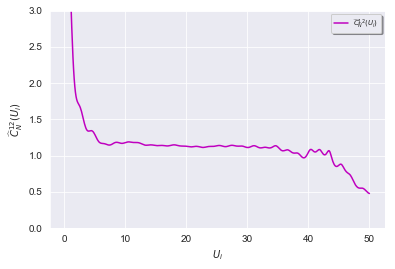}
		\caption{Covariance estimator}
		\label{fig:cov12}
	\end{subfigure}%
	\label{fig:FV10}
\end{figure}
Next, we consider some numerical experiments discussing how the Algorithms \ref{algo1} and \ref{algo2}, for the stopping rule, can be approximately implemented.  To illustrate the performance of the method for $ \widehat{U}_{start}^{oracle} $ in (\ref{uoracledata}) we proceed as follows. Fix $ r_{1} = 0.5, r_{2} = 1.5 $ therefore we assume at least one jump component is of infinite variation. Therefore, the co-jump index activity is given by $ r = 1.5 $. In Figures \ref{fig:os1} - \ref{fig:os2} we observe that the estimator is consistent to the true one $ C^{12} = 1 $ choosing $ \widehat{U}_{start}^{oracle} $ as is (\ref{uoracledata}) compare to Figures \ref{fig:cov1} - \ref{fig:cov12} where the behavior of the estimator is quite erratic in beginning of the estimation. This erratic behavior is explained because of the high bias at the beginning of the covariance estimation.
\begin{figure}[H]
	\begin{subfigure}[b]{0.5\textwidth}
		\centering
		\includegraphics[width=0.85\linewidth]{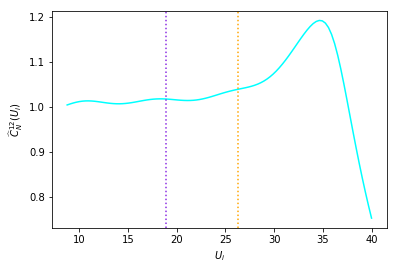}
		\caption{$n = 1,000$, $ r_{1} = 0.5, r_{2} = 1.5 $}
		\label{fig:os1}
	\end{subfigure}%
	\begin{subfigure}[b]{0.5\textwidth}
		\centering
		\includegraphics[width=.85\linewidth]{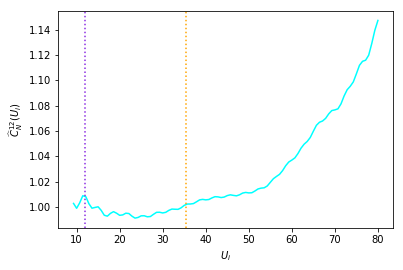}
		\caption{$ n = 5,000,  r_{1} = 0.5, r_{2} = 1.5 $}
		\label{fig:os2}
	\end{subfigure}%
	\caption{Vertical lines: purple-dashed is the $ \widehat{U}_{start}^{oracle} $, orange-dashed is the $ U_{bal} $. Blue curves: Adaptive estimator $ \widehat{C}^{12}_{n, j} $}
\end{figure}


\section{Proofs for \Cref{sec:characteristic}}\label{sec:proofs3}
In this section, we provide proofs of the results which are presented in Section \ref{sec:characteristic}.
The proof of Theorem \ref{ecf} follows a chaining argument for the empirical processes. Thus,
we recall the following definitions from empirical process theory.\par
\begin{definition}
	We consider measurable functions $ f,g:\mathcal{F}\to \R $. For two such functions $ f,g $ we introduce the ``bracket'' notation:
	\begin{equation}\label{bracket}
	[f, g]:= \big\{ h:\mathcal{F}\to \R \quad \mbox{such that}\quad f\leq h\leq g \big\}.
	\end{equation}
\end{definition}

\begin{definition}
	By the bracketing entropy number $ \ent $ of a class $ \g $ we mean the minimal number $ N $ for which there exist functions $ f_{1},\dots, f_{N} $ and $ g_{1}, \dots , g_{N}$ such that
	\begin{equation*}
	\g \subset \bigcup_{i= 1}^{N}[f_{i}, g_{i}] \qquad \mbox{and} \quad \int|f_{i}-g_{i}|^{2}d\Pm\leq \epsilon^{2}, \quad i = 1,2,\dots, N.
	\end{equation*}
\end{definition}

$ \ent $ is the minimal number of $ L^{2}(\Pm) $-balls of radius $ \epsilon $ which are needed to cover $ \g $. The class $ \g  $ is called bracketing compact if $\ent<\infty $ for any $ \epsilon>0 $. The entropy integral is defined by 
\begin{equation*}
\ei := \int_{0}^{\delta}\big( \log(\ent)\big)^{1/2} d\epsilon.
\end{equation*}
The convergence of the integral depends on the size of the bracketing numbers for $ \epsilon \to 0 $. Finally, a function $ F\geq 0$ is called an envelope function for $ \g $, if 
\begin{equation*}
\forall f  \in \g: |f|\leq F.
\end{equation*}
\subsection*{Proof of Theorem \ref{ecf}}
We decompose $ C_{n}$ in its real and imaginary parts,
\begin{equation*}
\begin{aligned}
\re(C_{n}(\textbf{u})) := n^{-1/2}\sum_{t=1}^{n}\big(\cos (\dotprod)- \E\cos(\dotprodu)\big)\\
\im(C_{n})(\textbf{u}) := n^{-1/2}\sum_{t=1}^{n}\big(\cos (\dotprod)- \E \cos(\dotprodu)\big)
\end{aligned}
\end{equation*}
We consider the class $ \g $, which consists of complexed valued functions,
\begin{equation*}
\g  = \Big\{\textbf{x}\to w(U)\cos(\langle \textbf{u}, \textbf{x}\rangle)| \textbf{u}\in \R^{2} \Big\} \cup \Big\{\textbf{x}\to w(U)\sin(\langle \textbf{u}, \textbf{x}\rangle)| \textbf{u}\in \R^{2} \Big\}.
\end{equation*}
An application of Corollary 19.35 in \cite{van2000asymptotic}  gives
\begin{equation}\label{van}
\E\|C_{n}\|_{L_{\infty}(w)} \leq C J_{[\cdot]}\big(\E[F^{2}(X_{1})], \g, L^{2}(\Pm)\big),
\end{equation}
where $ F = 1$ is an envelope function in $ \g $. It remains to prove that the bracketing integral on the right-hand side of (\ref{van}) is bounded. We need to cover $ \g  $ with functions such that 
\begin{equation*}
\g \subset \bigcup_{i= 1}^{N}\Big\{[g_{i}^{-}, g_{i}^{+}]: \int|g_{i}^{+}-g_{i}^{-}|^{2}d\Pm \leq \epsilon^{2}\Big\}
\end{equation*}
and find $ N $, the minimum number to cover $ \g $. Inspired by \cite{yukich1985weak}, we characterize the convergence of $ C_{n}(\textbf{u}) $ in terms of the tail behavior of $ \Pm $. For every $ \epsilon>0  $, we set
\begin{equation}
M := M(\epsilon):=\inf\{m>0: \Pm(|X_{1}|>m)\leq \epsilon^{2}\}.
\end{equation}
Furthermore, for all $ j $, define the bracket functions for $ \textbf{x} = (x_{1}, x_{2}) $
\begin{equation*}
\begin{aligned}
g_{j}^{\pm}(\textbf{x}) = \big( w(U_{j})\cos(\langle \textbf{u}_{j}, \textbf{x} \rangle) \pm \epsilon\big)\mathds{1} _{[-M, M]}(\textbf{x}) \pm \|w\|_{\infty} \mathds{1} _{[-M, M]^{\complement}}(\textbf{x}),\\
h_{j}^{\pm}(\textbf{x}) = \big( w(U_{j})\sin(\langle \textbf{u}_{j}, \textbf{x} \rangle) \pm \epsilon\big)\mathds{1} _{[-M, M]}(\textbf{x}) \pm \|w\|_{\infty} \mathds{1} _{[-M, M]^{\complement}}(\textbf{x}),
\end{aligned}
\end{equation*}
where $ \textbf{u}_{j} = (U_{j}, U_{j}) $ and $ \textbf{x} = (x_{1}, x_{2}) $. We obtain for the size of the brackets that

\begin{equation*}
\begin{aligned}
\E \big[ |g_{i}^{+}(X_{1})-g_{i}^{-}(X_{1})|^{2}\big] &\leq \E \big[ |2 \epsilon \mathds{1}_{[M, M]^{2}}(\textbf{x}) +2\|w\|_{\infty}\mathds{1}_{[-M, M]^{\complement}}(\textbf{x})|^{2} \big]\\
&\leq 4 \epsilon^{2}(1 +\|w\|_{\infty}^{2}).
\end{aligned}
\end{equation*}
An analagous argument gives
\begin{equation*}
\E \big[ |h_{i}^{+}(X_{1})-h_{i}^{-}(X_{1})|^{2}\big] \leq 4 \epsilon^{2}(1 +\|w\|_{\infty}^{2}). 
\end{equation*}
It remains to choose $U_{j} $ in such a way that the brackets cover $ \g $. We consider an arbitrary $U \in\R $ and any grid point $ U_{j} $. For a function $ g_{U}(\cdot):= w(U) \cos(U(x_{1}+x_{2})) \in \g $ to be contained in the bracket $[g^{-}_{j}, g_{j}^{+}] $, we have to ensure 
\begin{equation}\label{brac}
|w(U)\cos(\langle \textbf{u}_{j}, \textbf{x} \rangle) - w(U)\cos(\langle \textbf{u}, \textbf{x} \rangle)| \leq \epsilon, \quad \mbox{$\forall \textbf{x} \in [-M, M]^{2}$}.
\end{equation}
With the estimate 
\begin{equation*}
\begin{aligned}
&|w(U)\cos(\langle \textbf{u}_{j}, \textbf{x} \rangle) - w(U)\cos(\langle \textbf{u}, \textbf{x} \rangle)| \\
\leq & \big(w(U) +w(U_{j})\big) \wedge\\
&\big(|w(U)\cos(\langle \textbf{u}, \textbf{x}\rangle)  - w(U) \cos(\langle \textbf{u}_{j}, \textbf{x}\rangle)|\mathds{1}_{[M, M]^{2}}(\textbf{x})  \\
&+|w(U)\cos(\langle \textbf{u}_{j}, \textbf{x}\rangle)  - w(U_{j}) \cos(\langle \textbf{u}_{j}, \textbf{x}\rangle)|\mathds{1}_{[M, M]^{2}}(\textbf{x}) \big)\\
\leq &\big(w(U) +w(U_{j})\big) \wedge \big( 2M \|w\|_{\infty}|U - U_{j}| + Lip(w)|U- U_{j}|\big),
\end{aligned}
\end{equation*}
where $ Lip(w) $ is the Lipschitz constant of the weight function $ w $. In the last inequality, we used the estimate $ |\cos(\langle \textbf{u}, \textbf{x}\rangle) - \cos(\langle \textbf{u}_{j}, \textbf{x} \rangle) |\leq 2\big|\sin((x_{1}+x_{2})(U - U_{j})/2)\big|\leq 2M|U- U_{j} |$. (\ref{brac}) is seen to hold for any $ \textbf{u}\in \R^{2} $, when
\begin{equation*}
\min\big\{  w(U) +w(U_{j}), |U - U_{j}|\big(Lip(w)+ 2M\|w\|_{\infty}\big)\big\} \leq \epsilon.
\end{equation*}
Consequently, we choose the grid points $ U_{j} $ such as
\begin{equation}\label{u1}
\sup_{1\leq j\leq N} |U_{j} - U_{j-1}| \leq \frac{2 \epsilon}{ Lip(w)+ 2M\|w\|_{\infty} }
\end{equation}
and 
\begin{equation}\label{u2}
U_{j} = \frac{j\epsilon}{Lip(w)+2M\|w\|_{\infty}}
\end{equation}
for $ |j| \leq J(\epsilon) $, where $ J(\epsilon) $ is the smallest integer such that $ w(U_{1})\leq \frac{\epsilon}{2}, \cdots,\\ w(U_{J(\epsilon)})\leq \frac{\epsilon}{2} $. We need to find the smallest integer $ J(\epsilon) $ in order to cover $ \g $ with $ L^{2}(\Pm) $-balls of radius $ \epsilon $. This yields 
\begin{equation*}
J(\epsilon) \leq \frac{2U(\epsilon)(Lip(w)+2M\|w\|_{\infty})}{\epsilon},
\end{equation*}
with $ U(\epsilon)\leq U_{J(\epsilon)} $, where 
\begin{equation*}
U(\epsilon) := \inf\Big\{ U>0: w(U)\leq \frac{\epsilon}{2}  \Big\} = O(\exp(\epsilon^{-(1+1/2)^{-1}})).
\end{equation*}
Therefore, the minimal number of $ L^{2}(\Pm) $-balls of radius $ \epsilon $ satisfies $ \ent \leq 2(2J(\epsilon)+1)$. The generalized Markov inequality yields that 
\begin{equation*}
M(\epsilon) \leq \bigg(\frac{\E |X_{1}|^{2+\gamma}}{\epsilon^{2}}\bigg) ^{1/\gamma} = O(\epsilon^{-2/\gamma}).
\end{equation*}
The entropy number bracketing satisfies
\begin{equation}
\begin{aligned}
\log (N_{[\cdot]}(\epsilon, \g))&\leq \log(U(\epsilon) ) + \log\bigg(  \frac{Lip(w) + 2M\|w\|_{\infty}}{\epsilon} \bigg)\\
& = O(\epsilon^{-(1+1/2)^{-1}} + \log (\epsilon^{-1-2/\gamma}))= O(\epsilon^{-(1+1/2)^{-1}}).
\end{aligned}
\end{equation}
Thus, we have shown that 
\begin{equation*}
J_{[\cdot]}(\E[F^{2}(X_{1})], \g, L^{2}(\Pm))= \int_{0}^{1} \sqrt{\log(N_{[\cdot]}(\epsilon, \g))} d\epsilon <\infty.
\end{equation*}
This completes the proof.

\subsection* {Proof of Lemma \ref{countable}.}
The proof consists in checking the assumptions of Lemma \ref{talagrand}. We denote by $ X_{j}^{(\textbf{u})} := |e^{i\langle \textbf{u}, Y_{j}\rangle}|$, where $ Y_{j} $ are L\'evy increments. We trivially have
\[
\sup_{\textbf{u}\in \mathcal{A}} \var(X_{1}^{(\textbf{u})}):=\sup_{\textbf{u}\in \mathcal{A}}\var(|e^{i\langle \textbf{u}, Y_{1}\rangle}|) \leq 1 \quad \mbox{and} \quad \sup_{\textbf{u}\in \mathcal{A}}X_{1}^{\textbf{(u)}}\leq 1.
\]
We set $ S_{n}^{(\textbf{u})} := n^{-1}(\cf(\textbf{u})- \phi_{n}(\textbf{u})) $. By Theorem \ref{ecf} we have positive constant $ C<0 $:
\[
\E \bigg[\sup_{\textbf{u}\in \mathcal{A}}|\cf(\textbf{u})-\phi_{n}(\textbf{u})|\bigg]\leq Cn^{-1/2}(w(U))^{-1}.
\]
Therefore, we can apply Talagrand's inequality with $ R = 1 $, and $ v^{2}=1$. The claim now follows from inserting Lemma \ref{talagrand}.

\subsection* {Proof of Lemma \ref{uniform}.}
The proof can be based on the countable set of rational numbers. By continuity of the characteristic function and of $ w $, it carries over to the whole range of real numbers.\par
By Lemma \ref{countable} and setting 
\begin{equation}\label{kappa}
\kappa := t(\log n)^{1/2}n^{-1/2} - (1+\epsilon)Cn^{-1/2},
\end{equation}
for some $ \epsilon>0 $, we have
\begin{equation}
\begin{aligned}
&\Pm\bigg[\exists q \in \mathbb{Q}: |\cf(q)-\phi_{n}(q)|\geq t(\log n)^{1/2}(w(q))^{-1}n^{-1/2}\bigg]\\
& \leq \Pm\bigg[\sup_{q\in\mathbb{Q}} |\cf(q)-\phi_{n}(q)|\geq t(\log n)^{1/2}n^{-1/2}\bigg]\\
&\leq  \Pm\bigg[\sup_{q\in\mathbb{Q}} |\cf(q)-\phi_{n}(q)|\geq (1+\epsilon)\E\big[\sup_{q\in \mathbb{Q}}|\cf(q)-\phi_{n}(q)|\big] +\kappa\bigg]\\
&\leq 2\exp\bigg(-n\bigg(\frac{\kappa^{2}}{c_{1}}\wedge \frac{\kappa}{c_{2}}\bigg)\bigg).
\end{aligned}
\end{equation}
By definition of $ \kappa $ and for some constant $ C $ large enough we get
\begin{equation}
\begin{aligned}
&2\exp\bigg(-n\bigg(\frac{\kappa^{2}}{c_{1}}\wedge \frac{\kappa}{c_{2}}\bigg)\bigg)\\
&=2\exp\bigg(-\frac{\big(t(\log n)^{1/2}-(1+\epsilon)C\big)^{2}}{c_{1}}\bigg)\vee 2 \exp\bigg(-\frac{n^{1/2}\big(t(\log n)^{1/2}-(1+\epsilon)C\big)}{c_{2}}\bigg)\\
&\leq C\exp\bigg(-\frac{(t-\beta)^{2}}{c_{1}}\log n\bigg)= C n^{-\frac{(t-\beta)^{2}}{c_{1}}}.
\end{aligned}
\end{equation}
By the continuity of the characteristic function, we extend the above results from the rational numbers to real line. This completes the proof.
\subsection*{Proof of Lemma \ref{set}.}
The claim follows using Lemma \ref{uniform} and the choice of $ \kappa \geq 4(\sqrt{p c_{1}}+\beta)$, where $ \beta$ and $ c_{1} $ are the constants from Lemma \ref{uniform}. In particular, we have
\begin{equation}
\Pm\bigg[\mathcal{E}^{\complement}\bigg]\leq Cn^{-\frac{(\kappa /4- \beta)^{2}}{c_{1}}}\leq C n^{-p}.
\end{equation}
Using the same argument we get $ \Pm\bigg[\tilde{\mathcal{E}}^{\complement}\bigg]\leq Cn^{-\frac{(\kappa /4- \beta)^{2}}{c_{1}}}\leq C n^{-p}
$.This proves the claim.
\subsection*{Proof of Lemma \ref{nk}.}
We consider the following partition in the diagonal of the characteristic function $ \mathcal{A} = \mathcal{A}_{1}\cup \mathcal{A}_{2}$ with

\begin{equation}\label{a1}
\mathcal{A}_{1}= \bigg\{\textbf{u} \in \mathcal{A}: |\phi_{n}(\textbf{u})| \leq  \frac{3\kappa}{4}\bigg(\frac{\log n}{n}\bigg)^{1/2}(w(U))^{-1}\bigg\},
\end{equation}

\begin{equation}\label{a2}
\mathcal{A}_{2}= \bigg\{\textbf{u} \in \mathcal{A}: |\phi_{n}(\textbf{u})| > \frac{3\kappa}{4}\bigg(\frac{\log n}{n}\bigg)^{1/2}(w(U))^{-1}\bigg\},
\end{equation}

We analyze the deviation of the truncated estimator from the true one on each aforementioned sets and the event $ \mathcal{E} $. The event $ \mathcal{E}^{\complement} $ is negligible by Lemma \ref{set}, so it is enough to take into consideration only the event $ \mathcal{E} $. First, we consider $ \mathcal{A}_{1}$. For arbitrary $ \textbf{u} \in \mathcal{A}_{1} $, we get 
\begin{equation}\label{a33}
\begin{aligned}
&\bigg|\frac{1}{\tilde{\phi}_{n}(\textbf{u})} - \frac{1}{\phi_{n}(\textbf{u})} \bigg|^{2} \\ &=\frac{|\tilde{\phi}_{n}(\textbf{u})-\phi_{n}(\textbf{u})|^{2}}{|\tilde{\phi}_{n}(\textbf{u})|^{2}|\phi_{n}(\textbf{u})|^{2}}\mathds{1}\bigg(\bigg\{|\cf(\textbf{u})|>\frac{\kappa}{2}\bigg(\frac{\log n}{n}\bigg)^{1/2}(w(U))^{-1} \bigg\}\bigg)\\
&+\frac{|\tilde{\phi}_{n}(\textbf{u})-\phi_{n}(\textbf{u})|^{2}}{|\tilde{\phi}_{n}(\textbf{u})|^{2}|\phi_{n}(\textbf{u})|^{2}}\mathds{1}\bigg(\bigg\{|\cf(\textbf{u})|\leq\frac{\kappa}{2}\bigg(\frac{\log n}{n}\bigg)^{1/2}(w(U))^{-1} \bigg\}\bigg)\\
& := A_{1} +A_{2}.
\end{aligned}
\end{equation}
We will bound the quantities $ A_{1} $ and $A_{2}  $ separately. We start with the first term of the sum (\ref{a33}), which is $ A_{1} $. On the set $\bigg\{|\cf(\textbf{u})|>\frac{\kappa}{2}\bigg(\frac{\log n}{n}\bigg)^{1/2}(w(U))^{-1} \bigg\}$, using the Definition \ref{truncatedest} we get that $ |\tilde{\phi}_{n}(\textbf{u})| = |\widehat{\phi}_{n}(\textbf{u})| $. On the one hand, by Lemma \ref{set}, we have 
\begin{equation}\label{a1a}
\frac{|\tilde{\phi}_{n}(\textbf{u})-\phi_{n}(\textbf{u})|^{2}}{|\tilde{\phi}_{n}(\textbf{u})|^{2}|\phi_{n}(\textbf{u})|^{2}} \leq \frac{\frac{k^{2}}{4^{2}}\frac{\log n}{n}(w(U))^{-2}}{|\widehat{\phi}_{n}(\textbf{u})|^{2}|\phi_{n}(\textbf{u})|^{2}}\leq \frac{1}{4}\frac{1}{|\phi_{n}(\textbf{u})|^{2}}.
\end{equation}
We observe that using (\ref{a1}), we get
\begin{equation}\label{a1b}
\frac{1}{|\widehat{\phi}_{n}(\textbf{u})|}\leq \frac{2}{\kappa \big(\frac{\log n}{n}\big)^{1/2} (w(U))^{-1}}\leq \frac{3}{2}\frac{1}{|\phi_{n}(\textbf{u})|}.
\end{equation}
On the other hand, inserting (\ref{a1b}) into $ A_{1} $, it yields that
\begin{equation}\label{a1c}
\frac{|\tilde{\phi}_{n}(\textbf{u})-\phi_{n}(\textbf{u})|^{2}}{|\tilde{\phi}_{n}(\textbf{u})|^{2}|\phi_{n}(\textbf{u})|^{2}} = \frac{|\widehat{\phi}_{n}(\textbf{u}) - \phi_{n}(\textbf{u})|^{2}}{|\widehat{\phi}_{n}(\textbf{u})|^{2}|\phi_{n}(\textbf{u})|^{2}} \leq \frac{9 k^{2}}{16}\frac{\log n (w(U))^{-2}n^{-1}}{|\phi_{n}(\textbf{u})|^{4}}.
\end{equation}
Combining (\ref{a1a}) with (\ref{a1c}), we get
\begin{equation}\label{A1a}
A_{1}\leq \frac{9k^{2}}{16}\frac{\log n (w(U))^{-2}n^{-1}}{|\phi_{n}(\textbf{u})|^{4}}\wedge \frac{1}{4}\frac{1}{|\phi_{n}(\textbf{u})|^{2}}.
\end{equation}
Next, we study the second term of the sum (\ref{a33}), which is $ A_{2} $. On the set $ \bigg\{|\cf(\textbf{u})|\leq \frac{\kappa}{2}\bigg(\frac{\log n}{n}\bigg)^{1/2}(w(U))^{-1} \bigg\} $, along with the Definition \ref{truncatedest}, we have that $ |\tilde{\phi}_{n}(\textbf{u})| = \frac{\kappa}{2}(\frac{\log n}{n})^{1/2}w(U)^{-1} $. It yields that 
\begin{equation}\label{a2a}
\begin{aligned}
A_{2} \leq \frac{|\tilde{\phi}_{n}(\textbf{u}) - \phi_{n}(\textbf{u})|^{2}}{|\tilde{\phi}_{n}(\textbf{u})|^{2}|\phi_{n}(\textbf{u})|^{2}}&\leq \frac{\big|\frac{3\kappa}{4}\big(\frac{\log n }{n}\big)^{1/2}(w(U))^{-1}-\frac{\kappa}{2}\big(\frac{\log n }{n}\big)^{1/2}(w(U))^{-1}\big|^{2}}{|\tilde{\phi}_{n}(\textbf{u})|^{2}|\phi_{n}(\textbf{u})|^{2}}\\
&\leq \frac{1}{4}\frac{1}{|\phi_{n}(\textbf{u})|^{2}}.
\end{aligned}
\end{equation}
Then, (\ref{A1a}) and (\ref{a2a}) imply that
\begin{equation}\label{A1}
A_{1} + A_{2}\leq \frac{9k^{2}}{16}\frac{\log n (w(U))^{-2}n^{-1}}{|\phi_{n}(\textbf{u})|^{4}}\wedge \frac{1}{4}\frac{1}{|\phi_{n}(\textbf{u})|^{2}}.
\end{equation}
The last ingredient is to consider the set $ \mathcal{A}_{2} $. Using (\ref{a2}), it holds that
\begin{equation}
|\cf(\textbf{u})|\geq \big||\phi_{n}(\textbf{u})| - |\cf(\textbf{u})
-\phi_{n}(\textbf{u})|\big|\geq \frac{\kappa}{2} \bigg(\frac{\log n}{n}\bigg)^{1/2}(w(U))^{-1}.
\end{equation}
By Definition \ref{truncatedest}, the above inequality implies that $ |\tilde{\phi}_{n}(\textbf{u})| =  |\widehat{\phi}_{n}(\textbf{u})|$. On the one hand we have
\begin{equation}\label{a21}
\begin{aligned}
\bigg|\frac{1}{\tilde{\phi}_{n}(\textbf{u})} - \frac{1}{\phi_{n}(\textbf{u})} \bigg|^{2} &= \frac{|\tilde{\phi}_{n}(\textbf{u})-\phi_{n}(\textbf{u})|^{2}}{|\tilde{\phi}_{n}(\textbf{u})|^{2}|\phi_{n}(\textbf{u})|^{2}}\\
& = \frac{|\widehat{\phi}_{n}(\textbf{u})-\phi_{n}(\textbf{u})|^{2}}{|\widehat{\phi}_{n}(\textbf{u})|^{2}|\phi_{n}(\textbf{u})|^{2}}\\
& \leq \frac{1}{4}\frac{1}{|\phi_{n}(\textbf{u})|^{2}}.
\end{aligned}
\end{equation}
On the other hand, Lemma \ref{set} and (\ref{a2}) give
\begin{equation}\label{a22}
\begin{aligned}
|\widehat{\phi}_{n}(\textbf{u})|&\geq \big||\phi_{n}(\textbf{u})| - |\cf(\textbf{u})
-\phi_{n}(\textbf{u})|\big|\\
&\geq \big| |\phi_{n}(\textbf{u})| - \frac{\kappa}{4}\bigg(\frac{\log n}{n}\bigg)^{1/2}(w(U))^{-1}\big|\\
&\geq 2|\phi_{n}(\textbf{u})|.
\end{aligned}
\end{equation}
Consequently, by (\ref{a22})
\begin{equation}\label{a23}
\begin{aligned}
\frac{|\widehat{\phi}_{n}(\textbf{u})-\phi_{n}(\textbf{u})|^{2}}{|\widehat{\phi}_{n}(\textbf{u})|^{2}|\phi_{n}(\textbf{u})|^{2}}&\leq \frac{\frac{\kappa^{2}}{4^{2}}\big(\frac{\log n }{n}\big)(w(U))^{-2}}{4|\phi_{n}(\textbf{u})|^{4}}\\
&\leq \frac{\kappa^{2}}{4^{3}}\frac{\big(\frac{\log n }{n}\big)(w(U))^{-2}}{|\phi_{n}(\textbf{u})|^{4}},
\end{aligned}
\end{equation}
which concludes the proof.
\subsection*{Proof of Lemma \ref{unk}.}

To derive the desired upper bound we distinguish between two events, $ E $ and $ E^{\complement} $, which are defined as in Lemma \ref{set}. So, 
\begin{equation}\label{unidec}
\begin{aligned}
\E\bigg[\sup_{\textbf{u}\in \mathcal{A}}\bigg|\frac{1}{\tilde{\phi}_{n}(\textbf{u})}- \frac{1}{\phi_{n}( \textbf{u})}\bigg|^{2}\bigg] &= \E\bigg[\sup_{\textbf{u}\in \mathcal{A}}\bigg|\frac{1}{\tilde{\phi}_{n}(\textbf{u})}- \frac{1}{\phi_{n}( \textbf{u})}\bigg|^{2}\mathds{1}(E)\bigg]\\
&+ \E\bigg[\sup_{\textbf{u}\in \mathcal{A}}\bigg|\frac{1}{\tilde{\phi}_{n}(\textbf{u})}- \frac{1}{\phi_{n}( \textbf{u})}\bigg|^{2}\mathds{1}(E^{\complement})\bigg]. 
\end{aligned}
\end{equation}
First, we establish an upper bound for the first part of the right hand side sum. 
Lemma \ref{nk} on the event $ E $ yields that,
\begin{equation*}
\begin{aligned}
&E\bigg[\sup_{\textbf{u}\in \mathcal{A}}\frac{\bigg|\frac{1}{\tilde{\phi}_{n}(\textbf{u})}- \frac{1}{\phi_{n}(\textbf{u})}\bigg| ^{2}}{\frac{\log n (w(U))^{-2}n^{-1}}{|\phi_{n}(U)|^{4}}\wedge \frac{1}{|\phi_{n}(\textbf{u})|^{2}}}\mathds{1}(E)\bigg]\\
&\leq \E \bigg[\sup_{u\in \mathcal{A}}\frac{\frac{9\kappa^{2}}{16}\frac{\log n (w(U))^{-2}n^{-1}}{|\phi_{n}(\textbf{u})|^{4}}\wedge \frac{1}{4}\frac{1}{|\phi_{n}(\textbf{u})|^{2}}}{\frac{\log n (w(U))^{-2}n^{-1}}{|\phi_{n}(U)|^{4}}\wedge \frac{1}{|\phi_{n}(\textbf{u})|^{2}}}\mathds{1}(E)\bigg]\leq \frac{9\kappa^{2}}{16}.
\end{aligned}
\end{equation*}
On the contrary, the event $ E^{\complement} $ is negligible, using Lemma \ref{set} and this concludes the proof.

\subsection*{Proof of \Cref{emptotrue}.}
This is a direct consequence of the Proof of Lemma \ref{unk}. The statement of the corollary can be found in formulas (\ref{a1a}), (\ref{a2a}), and (\ref{a21}). 
\appendix
\section{Concentration inequalities}\label{app}
We restate, for the reader's convenience, the concentration inequalities which will be essential for our reasoning. The following lemmas are classical. Proofs can be found for example in \cite{markus} or \cite{dudley2014uniform}. First, we present the Bernstein inequality.

\begin{lemma}\label{bernstein}
	Let $ X_{1}, \cdots, X_{n} $ be complex valued i.i.d. random variables with $ \E[X_{i}] = 0 $ and $ S_{n} := \sum_{i=1}^{n}\big( X_{i} - \E [X_{i}]\big) $. Suppose that $ \|X_{1}\|_{\infty} \leq R$ for some $R<\infty $. Then, the following holds true for arbitrary $ k>0 $
	\begin{equation*}
	\Pm\big(|S_{n}|\geq k\big)\leq 2\exp\bigg(-\frac{k^{2}}{4(\var(S_{n})+ kR)}\bigg).
	\end{equation*}

\end{lemma}

Finally, we need the Talagrand inequality, which strengthens the classical Bernstein to countable sets of random variables.
\begin{lemma}\label{talagrand}
	Let $\mathcal{U}$ be some countable index set. For each $ u\in \mathcal{U} $, let $ X_{1}^{(u)}, \dots, X_{n}^{(u)}$ be i.i.d. complex valued random variables, defined on the same probability space, with $ \|X_{1}^{(u)}\|_{\infty} \leq R $ for some $ R<\infty $. Let $ v^{2}:= \var X_{1}^{(u)}$. Then, for arbitrary $ \epsilon>0 $, there are positive constants $ c_{1} $ and $ c_{2} = c_{2}(\epsilon) $ depending only on $ \epsilon $ such than for any $ k>0 $:
	\begin{equation}\label{tala}
	\Pm \bigg[\sup_{u\in \mathcal{U}} |S_{n}^{(u)}|\geq (1+\epsilon)\E\big[\sup_{u\in \mathcal{U}}|S_{n}^{(u)}|\big]+k\bigg]\leq 2 \exp \bigg(-n\bigg(\frac{k^{2}}{c_{1}v^{2}}\wedge \frac{k}{c_{2}R}\bigg)\bigg).
	\end{equation}
\end{lemma}
Lemma \ref{talagrand} is taken from \cite{massart2007concentration}, see formula (5.50) on page 170. From the arguments given therein, we derive that for $ \eta \in (0,1)$, we can take $ c_{1} = 4/(1-\eta^{2}) $ and $ c_{2} = 4\sqrt{2}(1/3+\epsilon^{-1})/\eta $.

\section*{Acknowledgement}
The author is very grateful to Markus Rei\ss\/, Josef Jan\'ak and Martin Wahl for stimulating comments and discussions.
\bibliography{bibad}{}
\bibliographystyle{plainnat}

\end{document}